\documentclass[12]{article}
\usepackage{amscd,amsfonts,amsmath,amssymb,amsthm,amstext,latexsym}
\usepackage{enumerate,pifont,epsf,graphicx}
\usepackage[english]{babel}
\usepackage[T1]{fontenc}
\usepackage{epsfig}
\usepackage[latin1]{inputenc}

\newtheorem{theorem}{Theorem}[section]
\newtheorem{lemma}[theorem]{Lemma}
\newtheorem{proposition}[theorem]{Proposition}

\newtheorem{remark}[theorem]{Remark}
\newtheorem{corollary}[theorem]{Corollary}
\newtheorem{definition}[theorem]{Definition}
\newtheorem{conjecture}[theorem]{Conjecture}

\newtheorem{claim}[theorem]{Claim}

\begin{document}

\begin{title}{Non-simply connected minimal planar domains in $\mathbb{H}^2
    \times\mathbb{R}$}
\end{title}

\begin{author}{Francisco Mart\'\i n\thanks{This research is partially
      supported by MEC-FEDER Grant no. MTM2007 - 61775 and a Regional
      J. Andaluc\'\i a Grant no. P09-FQM-5088.}  \and M. Magdalena
    Rodr\'\i guez$^*$}
\end{author}
\maketitle

\begin{abstract}
  We prove that any non-simply connected planar domain can be properly
  and minimally embedded in $\mathbb{H}^2\times\mathbb{R}$.  The examples that we
  produce are vertical bi-graphs, and they are obtained from the
  conjugate surface of a Jenkins-Serrin graph.
\end{abstract}

\section{Introduction}\label{sec:intro}
One of the most fruitful methods to obtain minimal surfaces in
$\mathbb{H}^2\times\mathbb{R}$ is by solving the Dirichlet Problem
for minimal graphs, with possibly infinite boundary values. This
method was originally introduced by H. Jenkins and Serrin \cite{JS}
for minimal graphs in $\mathbb{R}^3$, and extended to $\mathbb{H}^2
\times\mathbb{R}$ by B. Nelli and H. Rosenberg~\cite{ner2}, P.
Collin and H. Rosenberg~\cite{cor2}, and L. Mazet, H. Rosenberg and
the second author~\cite{marr1}.

In~\cite{ner2}, Nelli and Rosenberg also constructed vertical
catenoids and helicoids.  L. Hauswirth~\cite{hau1} generalized these
examples by studying all minimal surfaces foliated by horizontal
constant curvature curves. In this way, he obtained a 2-parameter
family of minimal Riemann-type surfaces, which have genus zero and
infinitely many ends.

Very recently, J. Pyo~\cite{pyo1}, F. Morabito and the second
author~\cite{moro1} have constructed minimal surfaces of genus zero
and finite total curvature. The method of construction in both
papers consists of three steps. First, one solves the Jenkins-Serrin
problem in a suitable geodesic polygonal domain with vertices $p_1,
\ldots, p_{2n}$, satisfying $p_{2 i-1}$ in $\mathbb{H}^2$ and $p_{2
i}$ in the infinite boundary of $\mathbb{H}^2$ (that we will denote
as $\partial_{\infty}\mathbb{H}^2.$) Secondly, one uses the
conjugation introduced by B. Daniel \cite{da2} and Hauswirth, R. Sa
Earp and E. Toubiana \cite{HST} to obtain a minimal graph bounded by
$n$ planar geodesics of the surface (not ambient geodesics in
$\mathbb{H}^2\times\mathbb{R}$), all of them at the same height. The
complete surface is obtained by doubling the previous graph using
Schwarz reflection principle with respect the horizontal slice that
contains the horizontal geodesics (see Figure \ref{fig:karcher-5}).
 \begin{figure}[htbp]
   \begin{center}
 \includegraphics[width=0.45\textwidth]{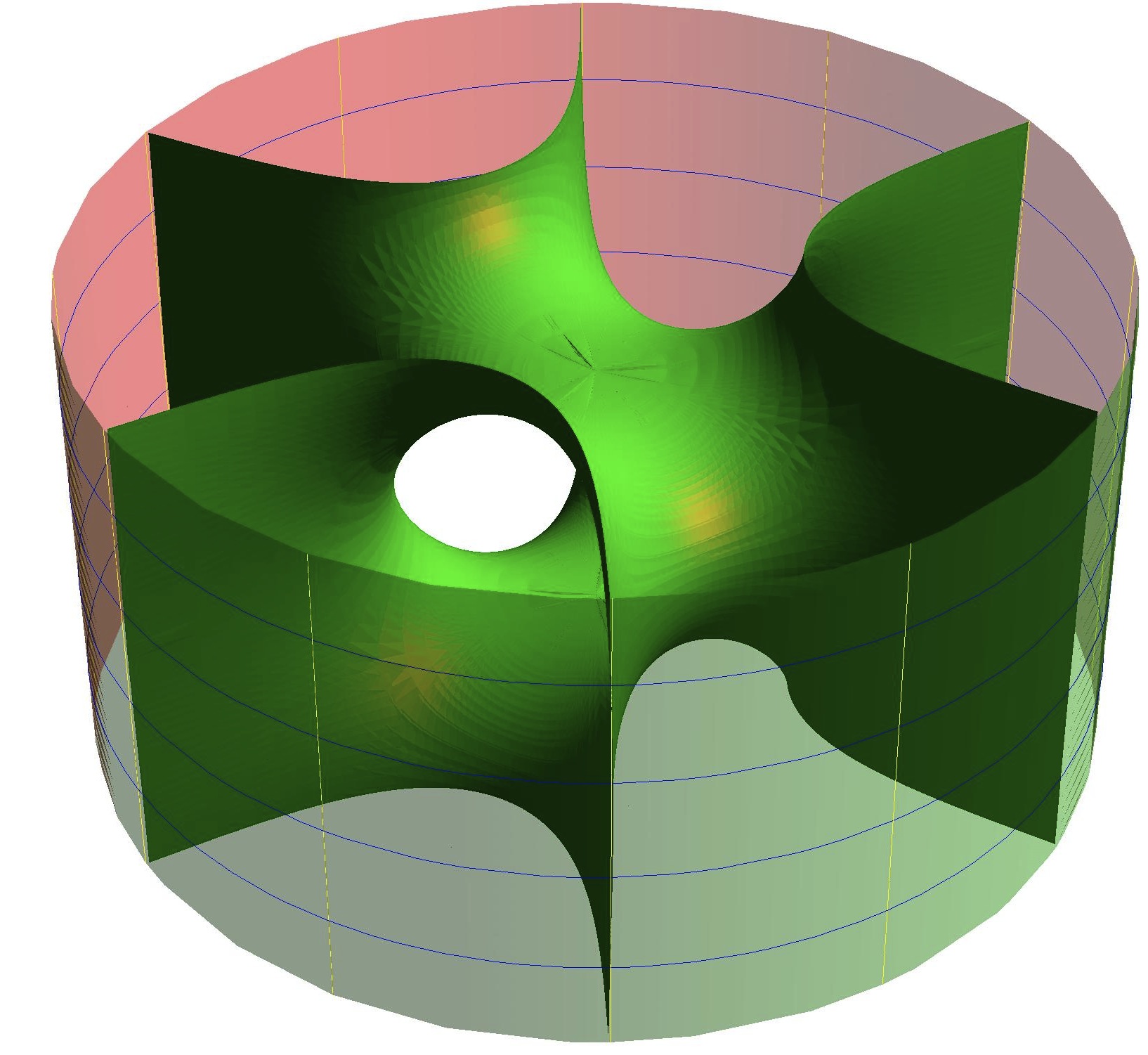}
 \end{center}
 \caption{\em \small One of the examples by Rodr\'\i guez and
   Morabito. It has the topology of a sphere minus three points.}
 \label{fig:karcher-5}
 \end{figure}

The main theorem of this paper shows that it is possible to take
limits in the method of construction described in the above
paragraph. Moreover, we have an important control of this limit
surface, in such a way we can prescribe the topology of the resulting
minimal surface. This control also allows us to guarantee that the
limit set of distinct ends are disjoint. Regarding the conformal
structure, the examples can be constructed with {\em parabolic}
conformal type. This is not rare, because in some sense the minimal
surfaces that we construct are limits of minimal surfaces with finite
total curvature.

So, the main result asserts:
\begin{quote} {\bf Theorem} \em Let $\Sigma$ be a non-simply connected
  planar domain. Then, there exists a proper minimal embedding $f
  :\Sigma \to\mathbb{H}^2 \times\mathbb{R}$. Furthermore, $f$ satisfies:
  \begin{enumerate}[(1)]
  \item $f(\Sigma)$ is a vertical bigraph symmetric with respect a
    horizontal slice.
  \item The annular ends of $f(\Sigma)$ are asymptotic to vertical
    planes.
  \item The embedding $f$ can be constructed so that for any two
    distinct ends $E_1$, $E_2$ of $\Sigma$, the limit sets $L(E_1)$,
    $L(E_2)$ in $\partial_\infty (\mathbb{H}^2 \times\mathbb{R})$ are disjoint.
  \item $f(\Sigma)$ has parabolic conformal type.
  \end{enumerate}
\end{quote}
The above theorem, which can be thought as a generalization of the
results in~\cite{R}, gives a partial answer to a more general
conjecture proposed to the authors by A. Ros:
\begin{conjecture}
  Let $M$ be an oriented open surface\footnote{ We say that a surface
    is {\bf open} if it is non-compact and without boundary.}, then
  $M$ can be properly embedded into $\mathbb{H}^2 \times\mathbb{R}$ as a minimal
  surface.
\end{conjecture}

Furthermore, the main theorem says to us that we cannot expect
classification theorems for properly embedded minimal surfaces in
$\mathbb{H}^2 \times\mathbb{R}$ just in terms of their topology,
like in $\mathbb{R}^3$. (Meeks, P\'erez and Ros recently proved in
\cite{mpr6} that the only planar domains properly embedded in
$\mathbb{R}^3$ are the plane, the catenoid, the helicoid and
Riemann's minimal surfaces.)

\section{Preliminaries}\label{sec:prel}
We consider the Poincar\'e disk model for the hyperbolic plane, i.e.
$$
\mathbb{H}^2=\{(x,y)\in\mathbb{R}^2\ |\ x^2+y^2<1\}
$$
with the hyperbolic metric $g_{-1}=\frac{4}{(1-x^2-y^2)^2} g_0$,
where $g_0$ is the Euclidean metric in $\mathbb{R}^2$, and let ${\bf
0}=(0,0)$ be the origin of $\mathbb{H}^2$. In this model, the
asymptotic boundary $\partial_\infty\mathbb{H}^2$ of $\mathbb{H}^2$
is identified with the unit circle $\{x^2+y^2=1\}$.

\subsection{The existence of simple exhaustions} \label{sec:simple}

In this paper we will use that any open orientable surface $M$ has a
smooth compact exhaustion $M_1\subset M_2\subset \cdots M_n\subset
\cdots$, called a \emph{simple exhaustion}, with the following
properties:
\begin{enumerate}[\bf \; \; \; 1.]
\item $M_1$ is a disk.
\item For any $n \in\mathbb{N}$, each component of $M_{n+1}-\mbox{Int}(M_n)$ has
  one boundary component in $\partial M_n$ and at least one boundary
  component in $\partial M_{n+1}$.
\item For any $n \in\mathbb{N}$, $M_{n+1}-\mbox{Int}(M_n)$ contains a unique
  non-annular component which topologically is a pair of pants or an
  annulus with a handle.
\end{enumerate}
If $M$ has finite topology with genus $g$ and $k$ ends, then we call
the compact exhaustion {\em simple} if properties 1 and 2 hold,
property 3 holds for $n\leq g+k$, and when $n> g+k$, all of the
components of $M_{n+1}-\mbox{Int}(M_n)$ are annular.

The reader should note that, for any simple exhaustion of $M$, each
component of $M-\mbox{Int}(M_n)$ is a smooth, non-compact proper
subdomain of $M$ bounded by a simple closed curve and for each $n
\in\mathbb{N}$, $M_n$ is connected (see Fig.~\ref{fig:simplex}).

\begin{figure}[htbp]
  \begin{center}
    \includegraphics[width=0.5\textwidth]{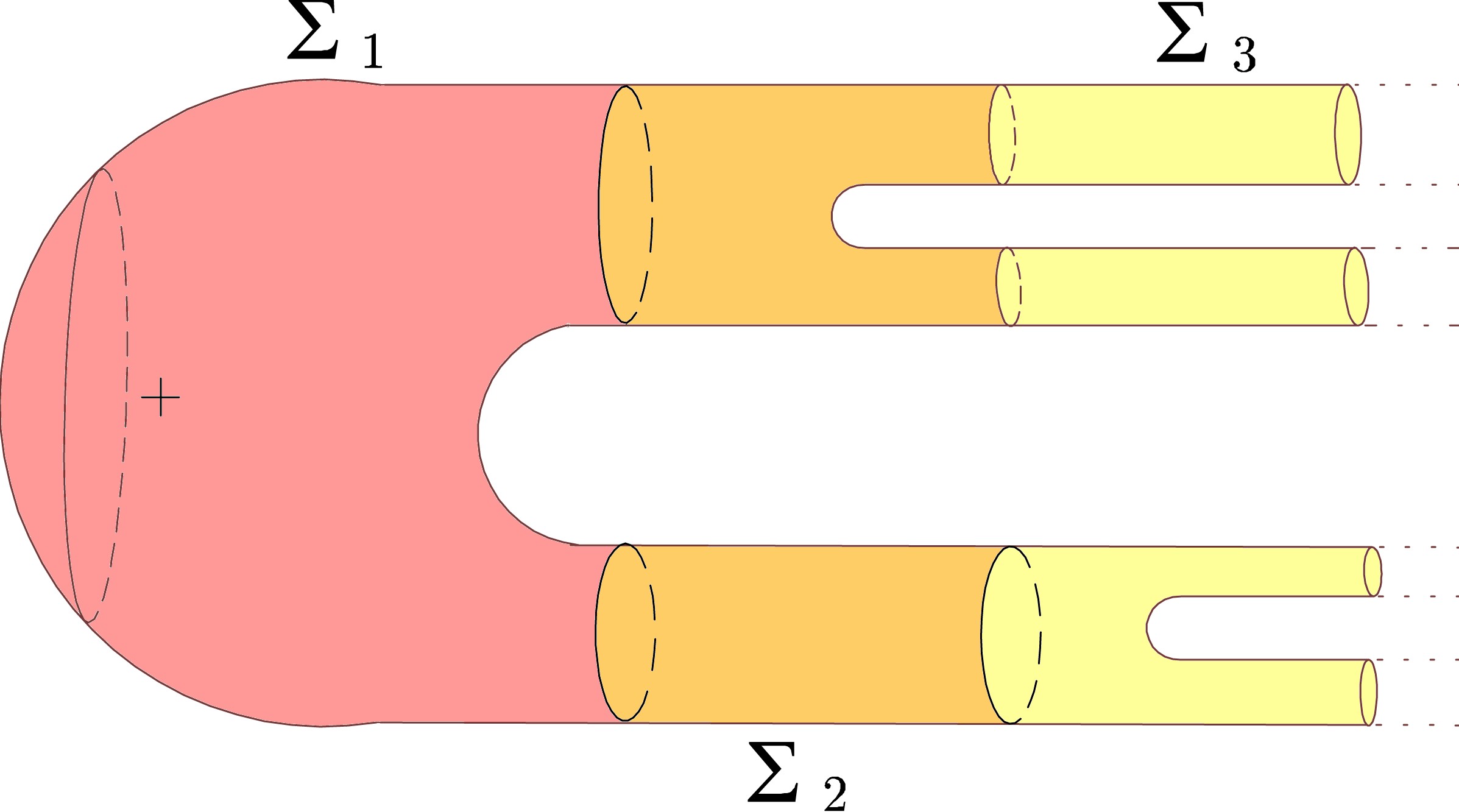}
  \end{center}
  \caption{\em \small A topological representation of the terms
    $\Sigma_1$ to $\Sigma_3$ in the exhaustion of an open surface $M$
    given in Lemma~\ref{lem:simple}.}
  \label{fig:simplex}
\end{figure}

In~\cite{fmm}, Ferrer, Meeks and the first author proved the following
result:

\begin{lemma}[\cite{fmm}] \label{lem:simple} Every orientable open
  surface admits a simple exhaustion.
\end{lemma}

A non-simply connected planar domain $\Sigma$ is a non-compact
orientable surface of genus $0$. As it has been mentioned in the
introduction, our main result is already known for minimal planar
domains with finite topology. Hence, we are going to focus on planar
domains with infinitely many ends.  In this case, Lemma
\ref{lem:simple} gives to us the following:

\begin{corollary}
  \label{co:simple}
  Let $\Sigma$ be a planar domain with an infinite number of
  ends. Then $\Sigma$ admits a compact exhaustion ${\cal S}=\{
  \Sigma_1 \subset \Sigma_2 \subset \cdots \}$, satisfying:
  \begin{enumerate}
  \item $\Sigma_1$ is a sphere minus two disks.
  \item Each component of $\Sigma_{n+1}-\mbox{Int}(\Sigma_n)$ has one
    boundary component in $\partial \Sigma_n$ and at least one
    boundary component in $\partial \Sigma_{n+1}$.
  \item $\Sigma_{n+1}-\mbox{Int}(\Sigma_n)$ contains a unique
    non-annular component which topologically is a { \bf pair of
      pants}.
  \end{enumerate}
\end{corollary}
We are also interested in the asymptotic behavior of the minimal
surfaces we are going to construct. So, we need some background about
the limit set of an end. In what follows, we will use the {\em ideal
  boundary} of $\mathbb{H}^2 \times\mathbb{R}$; $\partial_\infty (\mathbb{H}^2 \times\mathbb{R}) =
\left(\partial_\infty\mathbb{H}^2 \times\mathbb{R} \right) \cup
\left(\mathbb{H}^2\times
  \{\pm \infty\}\right).$
\begin{definition}\label{def:limit}
  Let $f\colon M \to\mathbb{H}^2 \times\mathbb{R}$ be a proper embedding of a
  surface $M$ with possibly non-empty boundary. The {\bf limit set} of
  $M$ is $$L(M)=\bigcap_{\alpha\in I}(\overline{f(M) - f(C_\alpha)}),$$ where
  $\{ C_{\alpha}\}_{\alpha \in I}$ is the collection of compact subdomains of
  $M$ and the closure $\overline{f(M) - f(C_\alpha)}$ is taken in
  $\partial_\infty(\mathbb{H}^2 \times\mathbb{R})$. The {\bf limit set $L(E)$ of an
    end $E$ of $M$} is defined to be the intersection of the limit
  sets of all properly embedded subdomains of $M$ with compact
  boundary which represent $E$. Notice that $L(M)$ and $L(E)$ are
  closed sets of $\partial_\infty(\mathbb{H}^2 \times\mathbb{R})$.
\end{definition}

\subsection{Minimal graphs}
Given an open domain $\Omega \subset {\mathbb H}^2 $ and a smooth
function $u:\Omega \to\mathbb{R}$, the graph surface of $u$ is
minimal in $\mathbb{H}^2\times\mathbb{R}$ when
\begin{equation}
  \label{eq.min.surf}
  {\rm div}\left( \frac {\nabla u}
    {\sqrt{1+|\nabla u|^2}} \right)=0,
\end{equation}
where all terms are calculated with respect to the metric of
$\mathbb{H}^2$.

\begin{definition}
  We say that a domain $\Omega\subset\mathbb{H}^2$ is {\it polygonal} when it
  is bounded by geodesic arcs. A polygonal domain $\Omega\subset\mathbb{H}^2$
  with a finite number of vertices (possibly at the infinite boundary
  $\partial_\infty\mathbb{H}^2$ of $\mathbb{H}^2$) is said to be {\it semi-ideal} when
  no two consecutive vertices are ideal (i.e. they are at
  $\partial_\infty\mathbb{H}^2$) nor interior (i.e. they lie in $\mathbb{H}^2$).
\end{definition}

Let $\Omega$ be a semi-ideal domain. In particular, $\Omega$ has an
even number of vertices $p_1,\ldots,p_{2k}$ (cyclically ordered),
with $p_{2i-1}\in\partial_\infty\mathbb{H}^2$ and
$p_{2i}\in\mathbb{H}^2$, for any $i=1,\ldots,k$. We call $A_i$
(resp. $B_i$) the geodesic arc joining $p_{2i-1},p_{2i}$ (resp.
$p_{2i},p_{2i+1}$); i.e.
\[
A_i=(p_{2i-1},p_{2i})_{\mathbb{H}^2}, \qquad
B_i=(p_{2i},p_{2i+1})_{\mathbb{H}^2} .
\]
We consider a horocycle $H_{2i-1}$ at each ideal vertex $p_{2i-1}$.
Assume $H_{2i-1}\cap H_{2j-1}=\emptyset$ for any $i\neq j$.  Given a
polygonal domain $\mathcal{P}$ inscribed in $\Omega$ (i.e. a
polygonal domain $\mathcal{P}\subset\Omega$ whose vertices are
vertices of $\Omega$, possibly at $\partial_\infty\mathbb{H}^2$), we
denote by $\Gamma(\mathcal{P})$ the part of $\partial \mathcal{P}$
outside the horocycles. (Observe that $\Gamma(\mathcal{P})=\partial
\mathcal{P}$ in the case all the vertices of $\mathcal{P}$ are in
$\mathbb{H}^2$.) Also let us call
\[
\alpha(\mathcal{P})=\sum_{i=1}^k\left|A_i\cap\Gamma(\mathcal{P})\right|
\qquad {\rm and}\qquad
\beta(\mathcal{P})=\sum_{i=1}^k\left|B_i\cap\Gamma(\mathcal{P})\right|
,
\]
where $|\bullet|=\mbox{length}_\mathbb{H}^2(\bullet)$.

\begin{definition}
  \label{Def:JS}
  {\rm A domain $\Omega\subset\mathbb{H}^2$ is called {\it admissible} when:
    \begin{enumerate}
    \item It is a convex semi-ideal polygonal domain with vertices
      $p_1,\ldots,p_{2k}$, with $p_{2i-1}\in\partial_\infty\mathbb{H}^2$ and
      $p_{2i}\in\mathbb{H}^2$.
    \item There exists a choice of disjoint horocycles $H_{2i-1}$ at
      the ideal vertices $p_{2i-1}$ such that:
      \begin{itemize}
      \item[(i)]
        $\mbox{dist}_{\mathbb{H}^2}(p_{2i-2},H_{2i-1})=\mbox{dist}_{\mathbb{H}^2}(p_{2i},H_{2i-1})$.
      \item[(ii)] $2\alpha(\mathcal{P})<\left|\Gamma(\mathcal{P})\right|$
        and $2\beta(\mathcal{P})<\left|\Gamma(\mathcal{P})\right|$, for
        every polygonal domain $\mathcal{P}$ inscribed in~$\Omega$,
        $\mathcal{P}\neq\Omega$.
      \end{itemize}
    \end{enumerate}
    Up to an isometry of $\mathbb{H}^2$, we can assume that the origin ${\bf
      0}=(0,0)$ is contained in $\Omega$.
    We say that $(\Omega,u)$ is an {\it admissible pair} if $\Omega$
    is an admissible domain and $u:\Omega\to\mathbb{R}$ is a solution to the
    minimal graph equation~\eqref{eq.min.surf} with $u({\bf 0})=0$ and
    whose boundary values are $+\infty$ on each edge $A_i$ and
    $-\infty$ on each $B_i$.  }
\end{definition}

We remark that condition (i) in the above definition does not depend
on the choice of horocycles; and if the inequalities of condition (ii)
are satisfied for some choice of horocycles, then they continue to
hold for ``smaller'' horocycles (see the argument given by Collin and
Rosenberg in~\cite{cor2}).

The following lemma is very useful to know when a domain satisfying
conditions 1 and 2-(i) in the above definition is admissible. We will
use this characterization in the proof of Lemma~\ref{lem:main}.

\begin{lemma}[\cite{R}]\label{lem:JS}
  Let $\Omega$ be a convex semi-ideal polygonal domain with vertices
  $p_1,\ldots,p_{2k}$, with $p_{2i-1}\in\partial_\infty\mathbb{H}^2$ and
  $p_{2i}\in\mathbb{H}^2$. Suppose there exists a choice of disjoint
  horocycles $H_{2i-1}$ at the ideal vertices $p_{2i-1}$ such that
  $\mbox{dist}_{\mathbb{H}^2}(p_{2i-2},H_{2i-1})=\mbox{dist}_{\mathbb{H}^2}(p_{2i},H_{2i-1})$. Then
  $\Omega$ is admissible if, and only if,
  $p_{2j}\in\mathbb{H}^2-\overline{D_{2i-1}}$ for any $i\neq j,j+1$, where
  $D_{2i-1}$ is the horodisk at $p_{2i-1}$ passing through $p_{2i-2}$
  and $p_{2i}$.
\end{lemma}

The following theorem says that, given an admissible domain, it
exists a unique solution $u:\Omega\to\mathbb{R}$ to the minimal
graph equation~\eqref{eq.min.surf} on $\Omega$ such that
$(\Omega,u)$ is an admissible pair.

\begin{theorem}[\cite{cor2,marr1,moro1}]
  \label{th:JS} Let $\Omega$ be an admissible domain with edges $A_1,
  B_1,\ldots,A_k, B_k$ (cyclically ordered).  Then there exists a
  solution $u$ for the minimal graph equation~\eqref{eq.min.surf} in
  $\Omega$ with boundary values
  \[
  u|_{A_i}=+\infty\quad\mbox{and}\quad u|_{B_i}=-\infty, \quad
  \mbox{for any }\ i=1,\ldots, k.
  \]
  This solution is unique up to an additive constant.

  Moreover, if we denote by $\Sigma^*$ the conjugate surface, then
  $\Sigma^*$ is a graph of a function $u^*$ over an ideal domain
  $\Omega^*$ with
  \[
  \partial \Omega^*=\gamma_1^* \cup \delta_1^* \cup \ldots \cup
  \gamma_k^* \cup \delta_k^*, \quad \mbox{(cyclically ordered),}
  \]
  where:
  \begin{enumerate}
  \item $\delta_1^*, \ldots, \delta_k^*$ are concave curves, with
    respect to $\Omega^*$,
  \item $u^*|_{\delta_i^*} =0$, for $i=1 , \ldots, k$,
  \item $\gamma_1^*, \ldots, \gamma_k^*$ are geodesics and
    $u^*|_{\gamma_i^*}=+\infty$, for any $i=1, \ldots,k,$
  \item $\delta_i^*$ is a horizontal geodesic curvature line of
    symmetry of $\Sigma^*$, for $i=1, \ldots,k$,
  \item $\delta^*_i$ and $\gamma^*_i$ (resp. $\delta^*_i$ and
    $\gamma_{i+1}^*$) are asymptotic at their common endpoint at
    $\partial_\infty\mathbb{H}^2.$
  \end{enumerate}
\end{theorem}
 \begin{figure}[htbp]
   \begin{center}
     \includegraphics[width=0.4\textwidth]{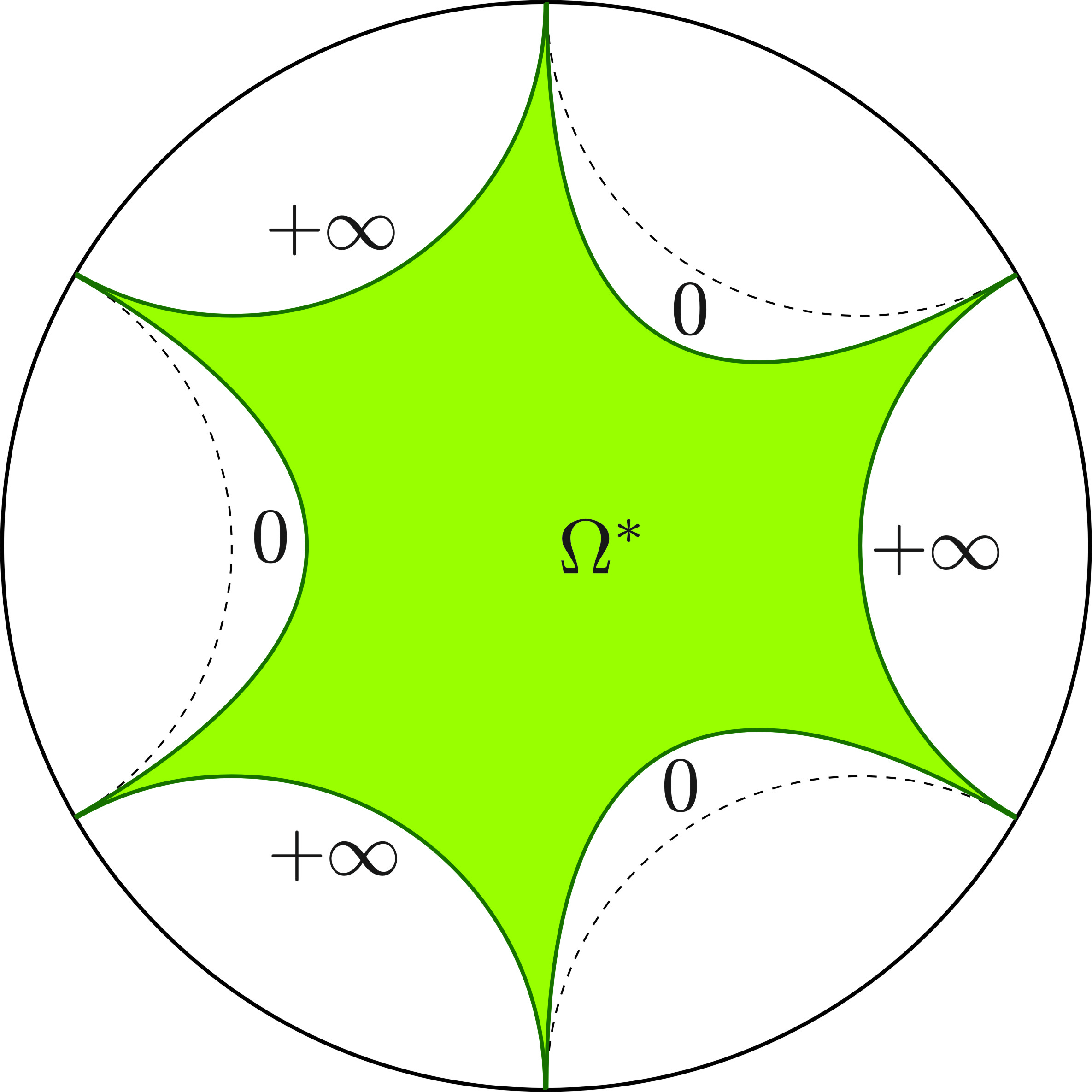}
     \hfill \includegraphics[width=0.5\textwidth]{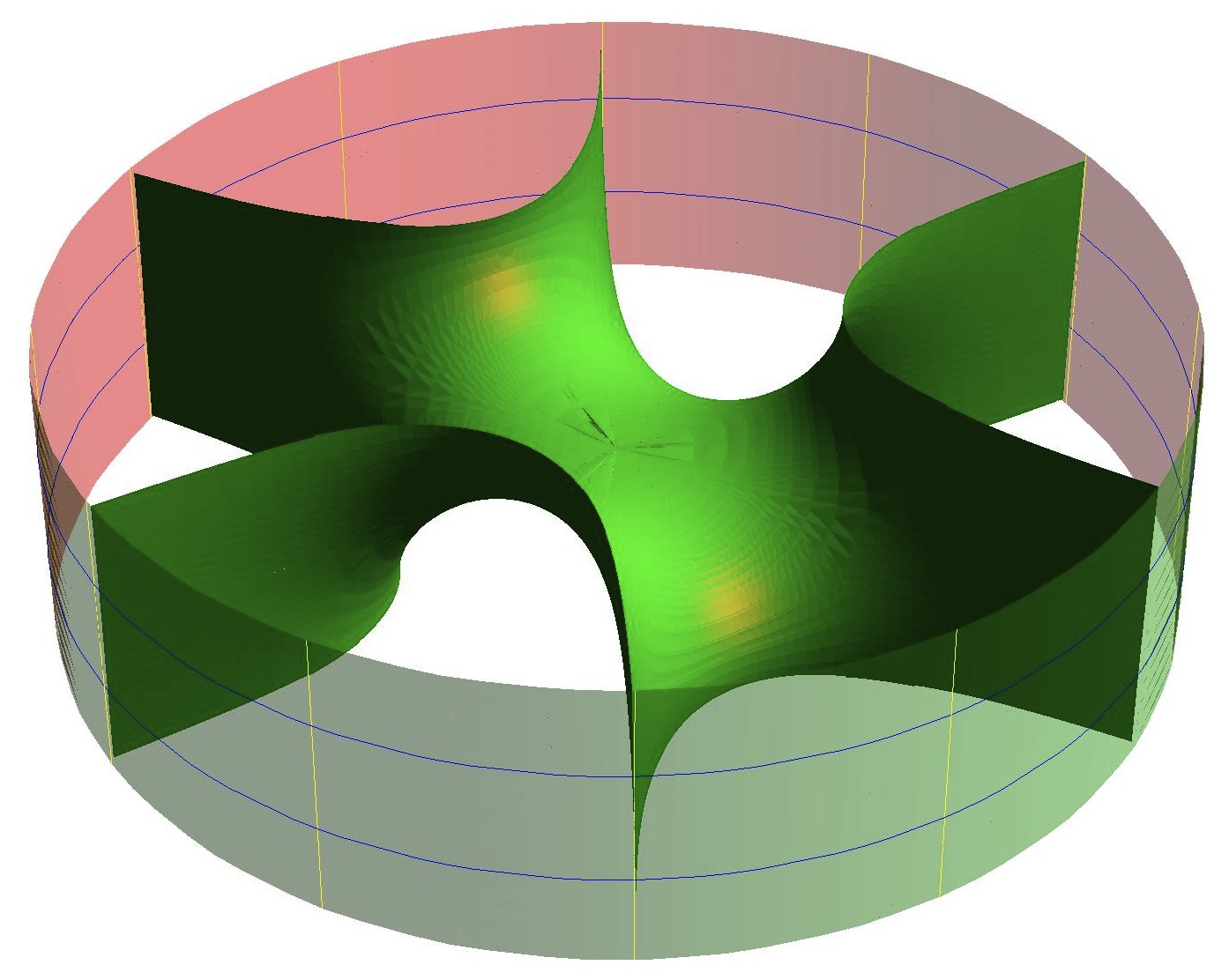}
   \end{center}
   \caption{\em \small Left: The domain $\Omega^*$. Right: The
     conjugate graph $\Sigma^*$.}
   \label{fig:karcher-1}
 \end{figure}

In the following subsections we present some useful tools used in the
proof of Theorem~\ref{th:JS}, which will also been used along the
present paper.

\subsubsection{Flux of a minimal graph along a curve}
Let $u$ be a minimal graph defined on a domain
$\Omega\subset\mathbb{H}^2$. Assume $\partial \Omega$ is piecewise
smooth and $u$ extends continuously to $\overline{\Omega}$ (possibly
with infinite values). We define the {\it flux} of $u$ along a curve
$\Gamma\subset\partial \Omega$ as
\[
F_u(\Gamma)= \int_\Gamma\left\langle\frac{\nabla u}{\sqrt{1+|\nabla
      u|^2}}, \eta\right\rangle ds,
\]
where $\eta$ is the outer normal to $\partial\Omega$ in
$\mathbb{H}^2$ and $ds$ is the arc-length of~$\partial\Omega$.

In the case $\Gamma\subset\Omega$, we can see $\Gamma$ in the
boundary of different subdomains of $\Omega$, with two possible
induced orientations. The flux $F_u(\Gamma)$ of $u$ along $\Gamma$
is then well-defined up to sign, and $|F_u(\Gamma)|$ is
well-defined.

\begin{lemma}[\cite{ner2}]\label{lem:flux}
  Let $u$ be a minimal graph on a domain $\Omega\subset\mathbb{H}^2$.
  \begin{itemize}
  \item[(i)] For every subdomain $\Omega'\subset\Omega$ such that
    $\overline{\Omega'}$ is compact, we have $F_u(\partial\Omega')=0$.
  \item[(ii)] Let $\Gamma$ be a piecewise smooth curve contained in
    the interior of $\Omega$, or a convex curve in $\partial\Omega$
    where $u$ extends continuously and takes finite values.  Then
    $|F_u(\Gamma)|<|\Gamma|$.
  \item[(iii)] If $T\subset\partial\Omega$ is a geodesic arc such that
    $u$ diverges to $+\infty$ (resp. $-\infty$) as one approaches $T$
    within $\Omega$, then $F_u(T)=|T|$ (resp. $F_u(T)=-|T|$).
  \end{itemize}
\end{lemma}

\begin{lemma}[\cite{marr1}]\label{lem:inf_flux}
  Let $u$ be a minimal graph on a domain $\Omega\subset\mathbb{H}^2$, and
  $T\subset\partial\Omega$ such that $|F_u(T)|=|T|$
  (resp. $|F_u(T)|=-|T|$). Then $u$ goes to $+\infty$
  (resp. $-\infty$) as we approach $T$ within $\Omega$.
\end{lemma}

\subsubsection{Divergence lines}
Let $\Omega\subset\mathbb{H}^2$ be a domain and $\{u_k\}_k$ a
sequence of minimal graphs on $\Omega$. We define the {\it
convergence domain} of $\{u_k\}_k$ as
\[
{\cal B}=\left\{p\in\Omega\ |\ \{|\nabla u_k(p)|\}_k \mbox{ is
    bounded}\right\},
\]
and the {\it divergence set} of $\{u_k\}_k$ as
\[
{\cal D}=\Omega-{\cal B}.
\]

The following proposition describes the convergence domain and the
divergence set of a sequence of minimal graphs.

\begin{proposition}[\cite{marr1}]\label{prop:div}
  Let $\Omega\subset\mathbb{H}^2$ be a domain and $\{u_k\}_k$ be a sequence of
  minimal graphs on~$\Omega$. Then:
  \begin{enumerate}
  \item ${\cal D}$ is composed of geodesic arcs contained in~$\Omega$
    (called divergence lines), each one joining two points of
    $\partial\Omega$ (including the vertices of $\Omega$).
  \item Let $L\subset {\cal D}$ be a divergence line. Passing to a
    subsequence, $|F_{u_k}(T)|\to|T|$ as $k\to+\infty$, for any
    geodesic arc $T\subset L$.
  \item If ${\cal D}=\emptyset$, then a subsequence of
    $\{u_k-u_k(p)\}_k$ converges uniformly on compact subsets of
    $\Omega$ to a minimal graph, for any $p\in\Omega$.
  \end{enumerate}
\end{proposition}

\subsection{Conjugate minimal surfaces}
\label{subsec:conj}

Let $\Sigma$ be a simply connected Riemann surface and
$X=(\varphi,h):\Sigma\rightarrow\mathbb{H}^2\times\mathbb{R}$ be a
conformal minimal immersion. It is known that $h$ is a real harmonic
function and $\varphi=\pi\circ X$ is a harmonic map from $\Sigma$
to~$\mathbb{H}^2$. Daniel~\cite{da2} and Hauswirth, Sa~Earp and
Toubiana~\cite{HST} proved that there exists a minimal immersion
$X^*=(\varphi^*,h^*):\Sigma\rightarrow
\mathbb{H}^2\times\mathbb{R}$, called {\it
  conjugate minimal immersion of $X$}, whose induced metric on
$\Sigma$ coincides with the one induced by $X$, and such that $h^*$
is the real harmonic conjugate function of $h$ and the Hopf
differential of $\varphi^*$ is $-Q_\varphi$, being $Q_\varphi$ be
the Hopf differential of $\varphi$. $X^*$ is well-defined up to an
isometry of $\mathbb{H}^2\times\mathbb{R}$.

If $N$ (resp. $N^*$) denotes the unit normal to $X$ (resp. $X^*$),
then $\langle N,\partial_t\rangle=\langle N^*,\partial_t\rangle$
(i.e. their angle maps coincide). Moreover, the correspondence $X
\leftrightarrow X^*$ maps:
\begin{itemize}
\item Vertical geodesics of $\mathbb{H}^2\times\mathbb{R}$ to horizontal geodesic
  curvature lines along which the normal vector field of the surface
  is horizontal.
\item Horizontal geodesics of $\mathbb{H}^2\times\mathbb{R}$ to geodesic curvature
  lines contained in vertical geodesic planes along which the normal
  vector field is tangent to the plane.
\end{itemize}

We will consider the conjugate surfaces of minimal graphs defined on
convex domains.  The surfaces obtained in this way are also minimal
graphs (and consequently embedded), as ensured by the following
Krust-type theorem given by Hauswirth, Toubiana and Sa Earp.

\begin{theorem}[\cite{HST}]\label{th:krust}
  If $\Sigma$ is a minimal graph over a convex domain $\Omega$ of
  $\mathbb{H}^2$, then $\Sigma^*$ is also a minimal graph over a
  (non-necessarily convex) domain of $\mathbb{H}^2$.
\end{theorem}

\section{Main Theorem}
Recall that the purpose of this paper is to show that any domain in
the plane which is not simply connected, can be properly embedded
into $\mathbb{H}^2 \times\mathbb{R}$ as a minimal bi-graph.  Since
this fact is known in the case of finite topology \cite{moro1,pyo1},
then we will focus throughout this section in the construction of
examples with infinite topology. The case of surfaces with an
uncountable number of ends will be particularly interesting.

The main tool in all this construction is Lemma~\ref{lem:main},
which gives us the approximation of an admissible pair by other
admissible pair with an extra ideal vertex. Its proof follows from
the ideas of Lemma 3.2 in~\cite{R}. Roughly speaking, this means
that we are able to increase the topology of the conjugate graph by
using surfaces which are close enough on compact regions. This kind
of ideas has been extensively used in the study of the Calabi-Yau
problem for minimal surfaces in $\mathbb{R}^3$.

Given an admissible pair $(\Omega,u)$, we call ${\cal V}_i(\Omega)$
the set of interior vertices of $\Omega$, and ${\cal
V}_\infty(\Omega)$ the set of its ideal vertices. We will finally
call ${\cal V}(\Omega)$ the set of vertices of $\Omega$, i.e.
\[
{\cal V}(\Omega)={\cal V}_i(\Omega)\cup{\cal V}_\infty(\Omega).
\]

\begin{figure}[htbp]
  \begin{center}
    \includegraphics[width=0.5\textwidth]{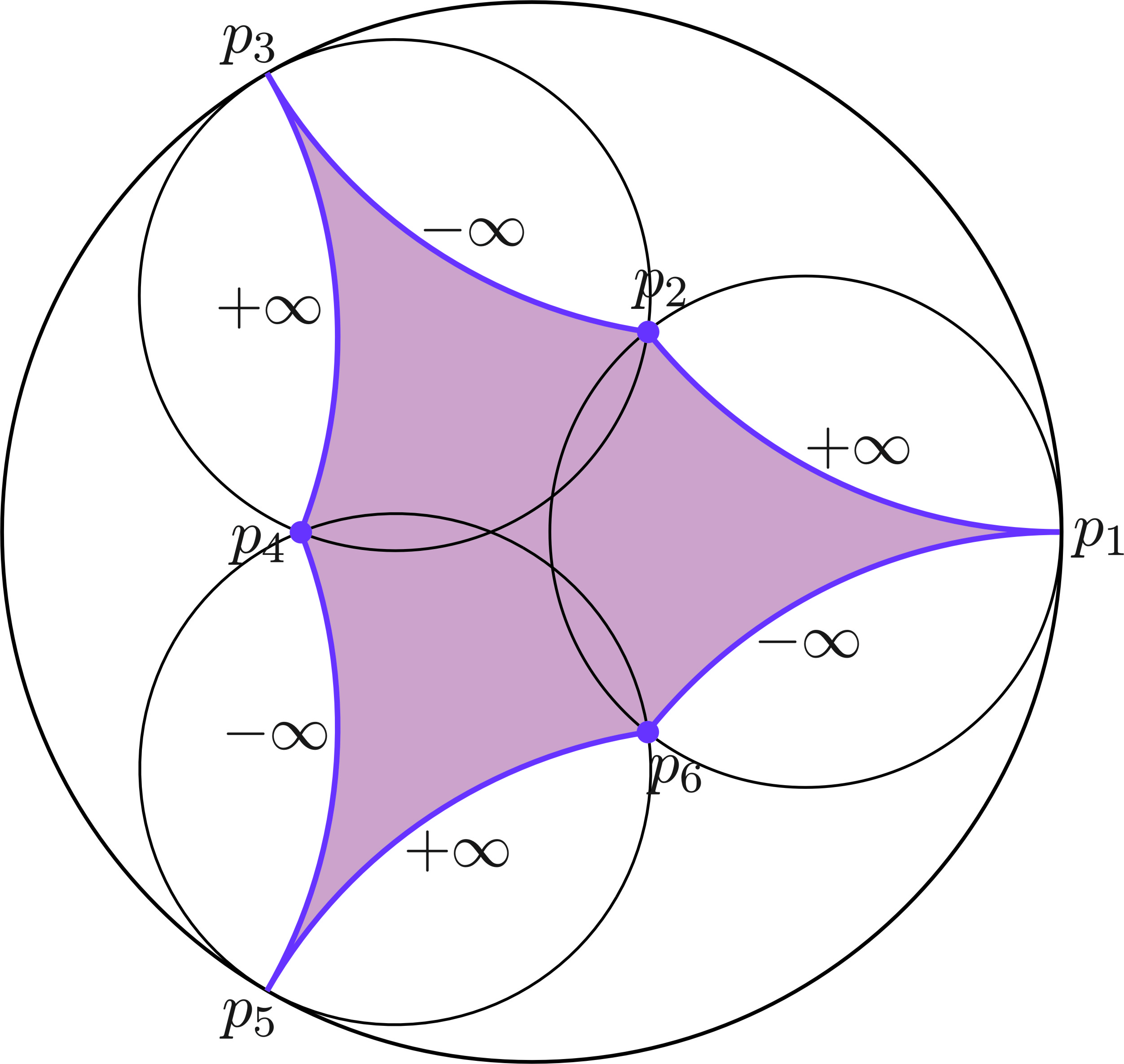}
  \end{center}
  \caption{}
  \label{fig:dominio0}
\end{figure}

\begin{lemma}
  \label{lem:main} Let $\varepsilon,\delta$ be positive numbers, and
  $(\Omega,u)$ an admissible pair. For any ideal vertex $P$ of
  $\Omega$ and any $R>0$ such that the hyperbolic disk $B(R)$ centered
  at $(0,0)$ of radius $R$ contains all the interior vertices
  of~$\Omega$, there exists an admissible pair
  $(\widetilde\Omega,\widetilde u)$ verifying:
  \begin{enumerate}
  \item Each boundary edge of $\Omega$ that does not have $P$ as an
    endpoint, is contained in the boundary of~$\widetilde\Omega$.  In
    particular, ${\cal V}(\Omega)-\{P\}\subset{\cal
      V}(\widetilde\Omega)$.
  \item $\widetilde\Omega$ only contains two ideal vertices and an
    interior vertex which are not vertices of $\Omega$; this is,
    ${\cal V}_\infty(\widetilde\Omega)-{\cal
      V}_\infty(\Omega)=\{P_1,P_2\}$ and ${\cal
      V}_i(\widetilde\Omega)-{\cal V}_i(\Omega)=\{P_0\}$.
  \item $\Omega\cap B(R)\subset\widetilde\Omega\cap B(R)$. In
    particular, $P_0\in\mathbb{H}^2-B(R)$.
  \item $\|\widetilde u-u\|_n<\varepsilon$ in $\Omega_\delta\cap B(R)$, for
    any $n \in\mathbb{N}$, where $\Omega_\delta=\{p\in\Omega\ |\
    \mbox{dist}_{\mathbb{H}^2}(p,\partial\Omega)\geq\delta\}$.
  \end{enumerate}
\end{lemma}

\begin{figure}[htbp]
  \begin{center}
    \includegraphics[width=0.5\textwidth]{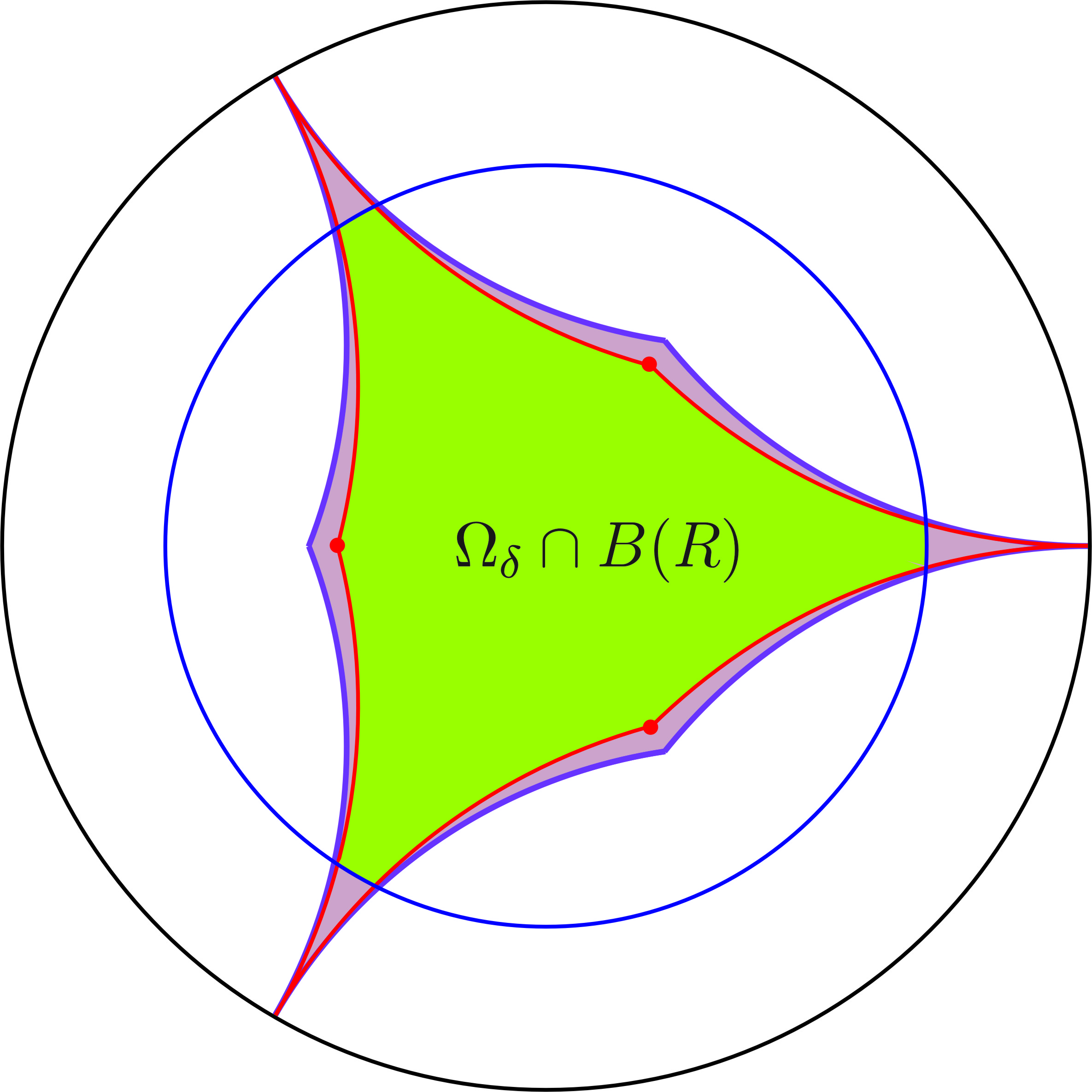}
  \end{center}
  \caption{}
  \label{fig:dominio3}
\end{figure}

\begin{proof}
  Up to an isometry of $\mathbb{H}^2$, we can assume $P=(1,0)$.  We call
  $p_1,p_2,\cdots,p_{2k}$ the vertices of $\Omega$, cyclically
  ordered, so that $p_1=P$.  We consider $P_n^+=e^{i/n},\
  P_n^-=e^{-i/n}$, for any $n\in\mathbb{N}$. It is clear that $P_n^\pm\to P$
  as $n\to+\infty$. We call $C_n^+$ (resp. $C_n^-$) the horocycle at
  $P_n^+$ (resp. $P_n^-$) passing through $p_2$ (resp. $p_{2k}$). For
  $n$ big enough, $C_n^+\cap C_n^-\neq\emptyset$. We call $P_n^0$ the
  intersection point in $C_n^+\cap C_n^-$ which is closer to $P$ (in
  the sense that the horodisk at $P$ passing through $P_n^0$ is
  contained in the horodisk at $P$ passing through the other point in
  $C_n^+\cap C_n^-$). We take $n$ big enough to assure
  $P_n^0\in\mathbb{H}^2-B(R)$.

  We call $p_1(n)=P_n^+$, $p_2(n)=p_2,\cdots,p_{2k}(n)=p_{2k}$,
  $p_{2k+1}(n)=P_n^-$, $p_{2k+2}(n)=P_n^0$, and $\Omega_n$ the
  polygonal domain with vertices $p_1(n),p_2(n),\cdots,p_{2k+2}(n)$.
  From the fact that $\Omega$ is an admissible domain and using that
  all the interior vertices of $\Omega_n$ remain fixed except for
  $p_{2k+2}(n)$, we can deduce that $\Omega_n$ is an admissible domain
  for $n$ large (here we use Lemma~\ref{lem:JS}). Let
  $u_n:\Omega\to\mathbb{R}$ be the solution to the minimal graph
  equation~\eqref{eq.min.surf} on $\Omega_n$ such that
  $(\Omega_n,u_n)$ is an admissible pair (it exists by
  Theorem~\ref{th:JS}).  It is clear that $\Omega_n\to\Omega$ as
  $n\to+\infty$. Let us prove that $u_n\to u$ uniformly on compact
  sets of $\Omega$.  By Proposition~\ref{prop:div}, it suffices to
  prove that the sequence $\{u_n\}$ does not have any divergence line.

  \begin{figure}[htbp]
    \begin{center}
     \includegraphics[width=0.5\textwidth]{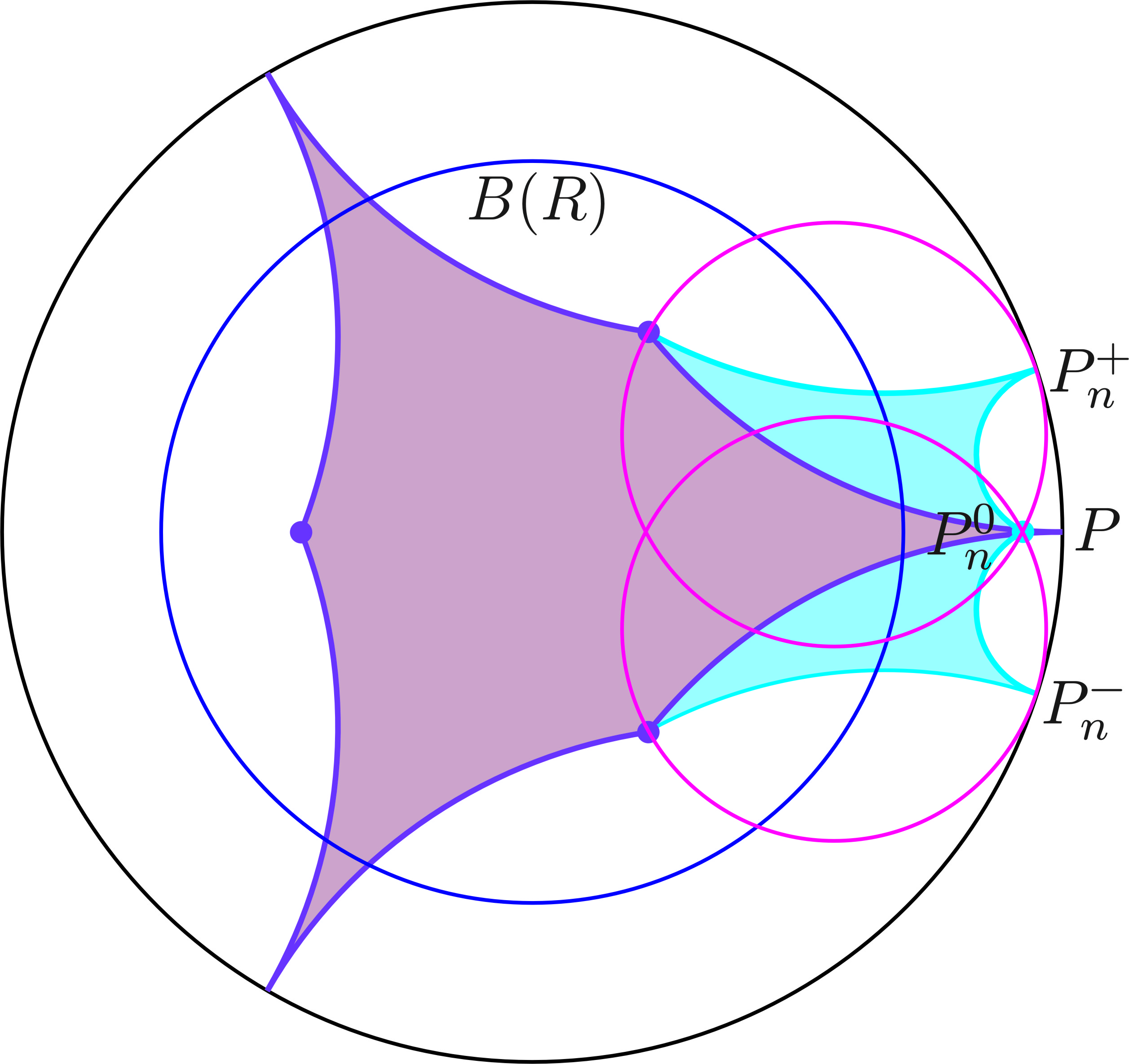}
    \end{center}
    \caption{}
    \label{fig:dominio1}
  \end{figure}

  Suppose by contradiction that $L\subset\Omega$ is a divergence line
  for $\{u_n\}$.  We call $L_n$ the intersection of $\Omega_n$ with
  the complete geodesic of $\mathbb{H}^2$ containing $L$. Since $\Omega_n$ is
  convex (by the choice of $P_n^0$), we get that $L_n$ is connected.
  Let ${\cal P}_n$ be a component of
  $\Omega_n-L_n$.

  For any $i=1,\cdots,k+1$, we call $D_{2i-1}(n)$ the open horodisk at
  $p_{2i-1}(n)$ passing through $p_{2i-2}(n),p_{2i}(n)$, and we
  consider a sequence of nested horocycles $H_{2i-1}(n,m)$ at
  $p_{2i-1}(n)$ contained in $D_{2i-1}(n)$ such that ${\rm
    dist}_{\mathbb{H}^2}(H_{2i-1}(n,m),\partial D_{2i-1}(n))=m$, for any $m$.
  In particular, for $m$ large we have $H_{2i-1}(n,m)\cap
  H_{2j-1}(n,m)=\emptyset$, if $i\neq j$. Let ${\cal P}_n(m)$ be the
  polygonal domain bounded by the part of $\partial{\cal P}_n$ outside
  the horocycles $H_{2i-1}(n,m)$, together with geodesic arcs joining
  the corresponding points in $\partial{\cal
    P}_n\cap\left(\cup_{i}H_{2i-1}(n,m)\right)$.  We also denote
  \[
  \alpha_n(m)= \sum_{i=1}^{k+1} |A^n_i\cap\partial{\cal P}_n(m)|,
    \qquad
    \beta_n(m)=\sum_{i=1}^{k+1} |B^n_i\cap\partial{\cal P}_n(m)|,
  \]
  \[
  f_n(m)=F_{u_n}(\partial{\cal P}_n(m)-\partial{\cal P}_n),
  \]
  where $A^n_i=(p_{2i-1}(n),p_{2i}(n))_{\mathbb{H}^2}$ and
  $B^n_i=(p_{2i}(n),p_{2i+1}(n))_{\mathbb{H}^2}$.  We observe that, for any
  fixed $n$, $|f_n(m)|<|\partial{\cal P}_n(m)-\partial{\cal P}_n|\to
  0$ as $m\to+\infty$.   We can choose ${\cal P}_n$ to have
  \[
  \beta_n(m)\geq \alpha_n(m).
  \]

  We consider similar definitions associated to $\Omega$: For any
  $i=1,\cdots,k$, let $D_{2i-1}$ be the open horodisk at $p_{2i-1}$
  passing through $p_{2i-2},p_{2i}$, and we consider a sequence of
  nested horocycles $H_{2i-1}(m)$ at $p_{2i-1}$ contained in
  $D_{2i-1}$ such that ${\rm dist}(H_{2i-1}(m),\partial D_{2i-1})=m$,
  for any $m$.

  We denote by $L(m)$ (resp. $L_n(m)$) the geodesic arc in $L$
  (resp. $L_n$) outside the horocycles $H_{2i-1}(m)$
  (resp. $H_{2i-1}(n,m)$).  By Lemma~\ref{lem:flux},
  \[
  F_{u_n}(L_n(m))=\beta_n(m)-\alpha_n(m)-f_n(m).
  \]
  We observe that $F_{u_n}(L_n(m))\geq 0$ for $m$ large.
  \begin{itemize}
  \item Suppose $L$ has finite length. Then $L$ joins a point
    $q_1\in[p_{2i},p_{2i+1})_{\mathbb{H}^2}\cup(p_{2i+1},p_{2i+2})_{\mathbb{H}^2}$ to
    a point
    $q_2\in[p_{2j},p_{2j+1})_{\mathbb{H}^2}\cup(p_{2j+1},p_{2j+2})_{\mathbb{H}^2}$,
    with $0\neq i\neq j$ (see Figure \ref{fig:dominio_LinDiv}).  We
    consider $m$ large enough so that $L(m)=L$ and $L_n(m)=L_n$.  The
    endpoints of $L_n$ are $q_1$ and another point that we are going
    to call $q_2(n)$ (notice that $q_2(n)=q_2$ when $j \neq 0$).  For
    $n$ large, one has $L\subset L_n$ and
    $|L_n|=|L|+\delta_n<+\infty$, where $\delta_n\geq 0$ converges to
    zero as $n\to +\infty$ ($\delta_n=0$ in the case
    $j\neq 0$).\\
    In this case, $c_n=\beta_n(m)-\alpha_n(m)$ does not depend on $m$ (it is
    also constant on $n$ when $j\neq 0$). Taking limits when $m$ goes
    to $+\infty$, we get $F_{u_n}(L_n)=c_n$. On the other hand,
    $|F_{u_n}(L_n)|\to|L|$ as $n\to+\infty$. Then $c_n\to|L|$. Let us
    see this is not possible. We call $C_1$ (resp. $C_2$, $C_2(n)$)
    the horocycle at $p_{2i+1}$ (resp. $p_{2j+1}$, $p_{2j+1}(n)$ )
    passing through $q_1$ (resp. $q_2$, $q_2(n)$), and
    \[
    d_1=\mbox{dist}(C_1,p_{2i}),\quad d_2=\mbox{dist}(C_2,p_{2j}),
    \quad d_2(n)=\mbox{dist}(C_2(n),p_{2j}(n)).
    \]
    We have that $|d_1-d_2(n)|=c_n$. Suppose $d_1>d_2$ (the case
    $d_2>d_1$ follows analogously). Thus, $d_1>d_2(n)$ for $n$ large
    enough.  Taking limits as $n \to \infty$ we have
    $d_1=|L|+d_2$. That implies that $p_{2j}$ (if
    $q_2\in[p_{2j},p_{2j+1})_{\mathbb{H}^2}$) or $p_{2j+2}$ (if
    $q_2\in(p_{2j+1},p_{2j+2})_{\mathbb{H}^2}$) lies on $D_{2i+1}$, a
    contradiction with the fact that $\Omega$ is admissible (see
    Lemma~\ref{lem:JS}).

    \begin{figure}[htbp]
      \begin{center}
        \includegraphics[width=0.5\textwidth]{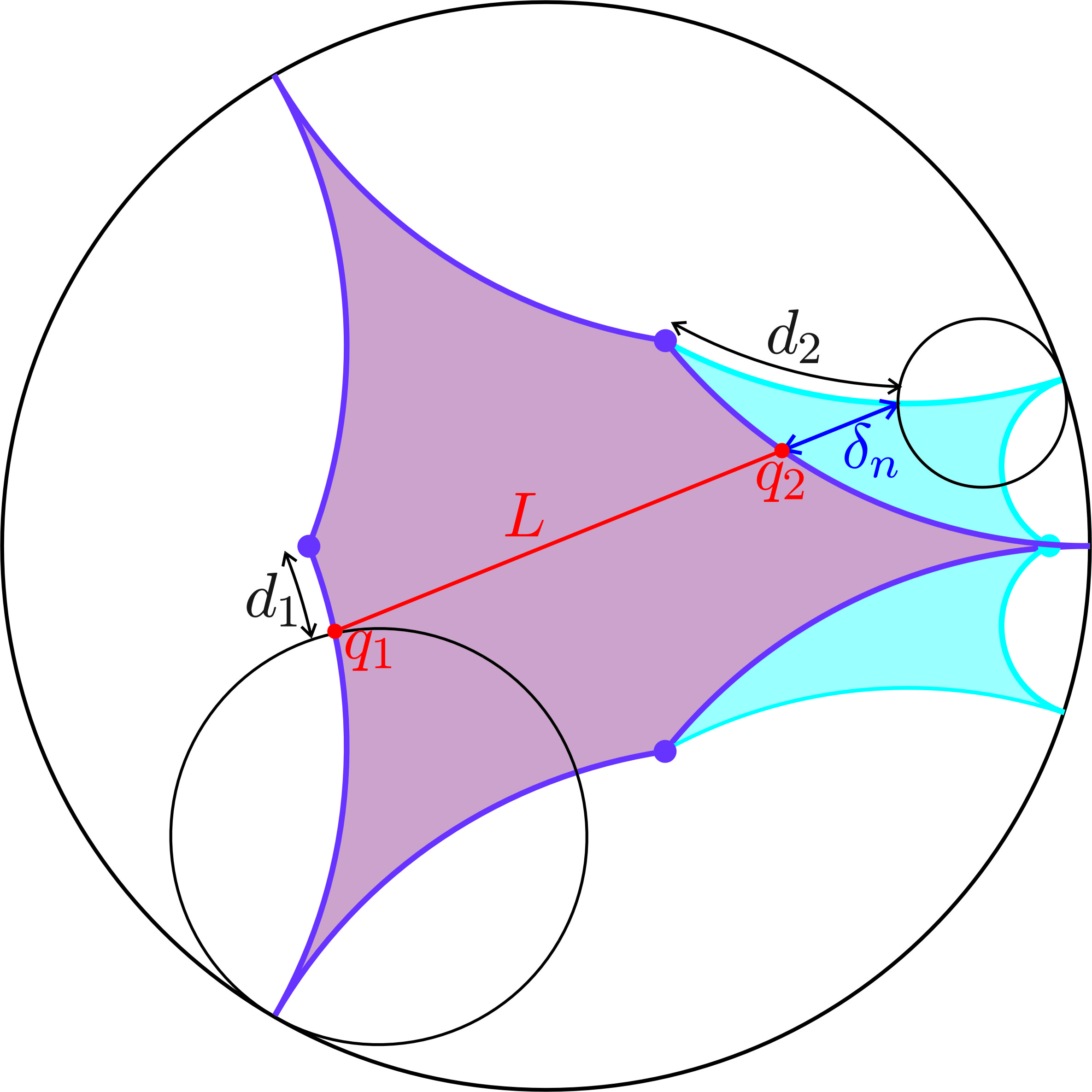}
      \end{center}
      \caption{}
      \label{fig:dominio_LinDiv}
    \end{figure}

  \item Now we suppose that $L$ joins an ideal vertex $p_{2i+1}$,
    $i\neq 0$, to
    $q\in[p_{2j},p_{2j+1})_{\mathbb{H}^2}\cup(p_{2j+1},p_{2j+2})_{\mathbb{H}^2}$, with
    $j\neq i$.  It holds $|L(m)|=m+d$, for some constant $d\in\mathbb{R}$. And
    for $n$ large, $L\subset L_n$ and $|L_n(m)|=|L(m)|+\delta_n$, with
    $\delta_n\geq 0$ converging to
    zero. \\
    On the other hand, for $m$ large we have that
    $c_n=m+\alpha_n(m)-\beta_n(m)\geq 0$ is constant on $m$ ($c_n=0$ when
    $q=p_{2j}$).  Then
    \[
    |L(m)|-|F_{u_n}(L(m))|=d+c_n+f_n(m)+\lambda_n,
    \]
    where $\lambda_n=F_{u_n}(L_n(m)-L(m))$ converges to zero as
    $n\to+\infty$. Since $|L(m)|-|F_{u_n}(L(m))|\to 0$ as
    $n\to+\infty$, we conclude that $c_n\to -d$. That implies that
    $p_{2j}\in D_{2i+1}$, if $q\in[p_{2j},p_{2j+1})_{\mathbb{H}^2}$, or
    $p_{2j+2}\in D_{2i+1}$, if $q\in(p_{2j+1},p_{2j+2})_{\mathbb{H}^2}$, a
    contradiction.

  \item We consider now that $L$ joins two ideal vertices
    $p_{2i+1},p_{2j+1}$, with $i\neq j$ both different from zero. Then
    we have $\alpha_n(m)=\beta_n(m)$ because of the choice of horocycles
    above.  For any compact geodesic arc $T\subset L_n$ and $m$ large,
    we have $|F_{u_n}(T)|\leq |F_{u_n}(L_n(m))|=|f_n(m)|$. Taking
    $m\to+\infty$, we get $F_{u_n}(T)=0$. But this contradicts that
    $|F_{u_n}(T)|\to|T|$ as $n\to+\infty$.

  \item If $L$ joins $p_1$ to another ideal vertex $p_{2i+1}$, $i\neq
    0$, then $L_n\subset L$ for any $n$. We have
    $\beta_n(m)-\alpha_n(m)=m-c_n$, with $c_n\geq 0$ independent of $m$
    ($c_n=0$ when $L_n$ finishes at $p_{2k+2}(n)$), and
    $|L_n(m)|=m+\delta_n$, where $\delta_n\in\mathbb{R}$.  Then,
    \[
    |L_n(m)|-|F_{u_n}(L_n(m))|=\delta_n+c_n+f_n(m)\to\delta_n+c_n,\quad\mbox{as
    } m\to+\infty.
    \]
    Since $|L_n(m)|-|F_{u_n}(L_n(m))|\to 0$ as $n\to+\infty$, we
    conclude that $\delta_n+c_n\to 0$. That implies that, for $n$ big
    enough, $p_{2k+2}(n)\in D_{2i+1}$, a contradiction, as $\Omega_n$
    is admissible.

  \item Finally, let us consider that $L$ joins $p_1$ to a point
    $q\in[p_{2j},p_{2j+1})_{\mathbb{H}^2}\cup(p_{2j+1},p_{2j+2})_{\mathbb{H}^2}$, with
    $j\neq 0$ (excluding the case $q=p_2, p_{2k}$). In this case we
    have $L_n\subset L$, $|L_n|<+\infty$ and $|L_n|=|L_n(m)|$ for big
    $m$.  When $n\to+\infty$, $|L_{n}(m)|-|F_{u_n}(L_{n}(m))|\to 0$,
    for any $m$. On the other hand,
    $|L_{n}(m)|-|F_{u_n}(L_n(m))|=|L_{n}|-|F_{u_n}(L_n)|\to
    |L_{n}|-c_n$ as $m\to+\infty$, where $c_n=\beta_n(m)-\alpha_n(m)$ for
    any $m$. The only possibility is $|L_{n}|-c_n\to 0$ as
    $n\to+\infty$. That contradicts the fact that $|L_{n}|\to+\infty$
    when $n\to+\infty$ while $c_n$ remains bounded.
  \end{itemize}

  Then we get that $\{|\nabla u_n|\}_n$ is uniformly bounded on
  compact sets of $\Omega$. Then Lemma~\ref{lem:main} holds for
  $(\widetilde\Omega,\widetilde u)=(\Omega_{n_0},u_{n_0})$ with some
  $n_0$ big enough, taking $P_0=P_{n_0}^0$, $P_1=P_{n_0}^+$ and
  $P_2=P_{n_0}^-$.
\end{proof}

Using Lemma~\ref{lem:main} we are able to prove the main result of
this paper.

\begin{theorem}
  \label{th:main}
  Let $\Sigma$ be a non-simply connected planar domain. Then, there
  exists a proper minimal embedding $f :\Sigma \to\mathbb{H}^2 \times\mathbb{R}$. Furthermore, $f$ satisfies:
  \begin{enumerate}[(1)]
  \item $f(\Sigma)$ is a vertical bigraph, symmetric with respect a
    horizontal slice.
  \item The annular ends of $f(\Sigma)$ are asymptotic to vertical
    planes.
  \item The embedding $f$ can be constructed so that for any two
    distinct ends $E_1$, $E_2$ of $\Sigma$, the limit sets\footnote{
      See Definition~\ref{def:limit} for the definition of the limit
      set of an end of a surface in a three-manifold. Recall that
      $\partial_\infty(\mathbb{H}^2\times\mathbb{R})=(\partial_\infty\mathbb{H}^2 \times\mathbb{R})
      \cup (\mathbb{H}^2\times \{\pm \infty\})$.} $L(E_1)$, $L(E_2)$ in
    $\partial_\infty (\mathbb{H}^2 \times\mathbb{R})$ are disjoint.
  \end{enumerate}
\end{theorem}
\begin{proof}
  In what follows, we are going to assume that $\Sigma$ has an
  infinite number of ends. Otherwise, we refer to \cite{moro1} .  From
  Corollary~\ref{co:simple}, the domain $\Sigma$ admits a simple
  exhaustion $ \{ \Sigma_1 \subset \Sigma_2 \subset \cdots \subset
  \Sigma_n \subset \cdots \} .$ We are going to give a labeling of the
  boundary components of the simple exhaustion that will give us a
  description of the set of ends of $\Sigma.$

  The boundary components of $\Sigma_1$ will be denoted by
  $\partial_0$ and $\partial_1$.  The difference $\Sigma_2 \setminus
  \Sigma_1$ consists of a pair of pants $P_2$ and a cylinder $C_2$.
  If the cylinder has $\partial_i$ as a common boundary with
  $\Sigma_1$ then we denote as $\partial_{i,0}$ to the other boundary
  component of $C_2$. On the other hand, if $\partial_j$ is the
  boundary component of $P_2$ that touches $\Sigma_1$, then we label
  $\partial_{j,0}$ and $\partial_{j,1}$ as the other two boundary
  components of $P_2$.

  Now, assume we have already labeled the boundary components of
  $\Sigma_n$. We are going to label the connected components of
  $\partial \Sigma_{n+1}$. We know that $\Sigma_{n+1} \setminus
  \Sigma_n$ consists of cylinders $C_{n+1}^1, \ldots,C_{n+1}^k$ and
  just one pair of pants $P_{n+1}$.  For a cylinder $C_{n+1}^i$, if
  the boundary component of $C_{n+1}^i$ which touches $\Sigma_n$ is
  labeled as $\partial_{i_1, \ldots,i_n}$, then we represent by
  $\partial_{i_1, \ldots,i_n,0}$ the other boundary component. In the
  case of the pair of pants $P_{n+1}$, if the boundary component of
  $P_{n+1}$ which touches $\Sigma_n$ is labeled as $\partial_{j_1,
    \ldots,j_n}$, then we denote by $\partial_{j_1, \ldots,j_n,0}$ and
  $\partial_{j_1, \ldots,j_n,1}$ the other two connected components of
  $\partial P_{n+1}$.

  At this point, we are going to construct a sequence of admissible
  pairs $(\Omega_n,u_n)$, where $\Omega_n$ is an admissible domain
  with $2(n+1)$ edges, and a sequence of radius $\{ R_n \}_{n \geq 2}$
  and positive constants $\{\varepsilon_n\}_{n \geq 2}$,
  $\{\delta_n\}_{n \geq 2}$, satisfying:

  \begin{enumerate}[(a)]
  \item $\varepsilon_n,\delta_n\in(0,1/2^n)$. In particular, $\displaystyle
    \sum_{n \geq 2} \varepsilon_n < +\infty$, and $\displaystyle
    \sum_{n \geq 2} \delta_n < +\infty$.
  \item $\Omega_{n+1}$ contains all the vertices of $\Omega_n$, except
    for and ideal vertex $p$.
  \item $\Omega_{n+1}$ only contains two ideal vertices and an
    interior vertex which are not vertices of $\Omega_n$.  In
    particular, each boundary edge of $\Omega_{n}$ that does not
    contain $p$, is contained in $\partial\Omega_{n+1}.$
  \item $\Omega_n \cap B(R_{n+1}) \subset \Omega_{n+1} \cap
    B(R_{n+1}).$
  \item For any $k \in\mathbb{N}$, we have $\|u_{n+1}-u_n\|_k
    <\varepsilon_{n+1}$ in the domain
    $\Delta_n\stackrel{\rm def}{=}\Omega_n(\delta_{n+1}) \cap B(R_{n+1}) $.  We recall
    that $\Omega_n(\delta_{n+1})= \{p \in \Omega_n\ |\
    \mbox{\rm dist}_{\mathbb{H}^2}(p, \partial \Omega_n)>\delta_{n+1}\}.$
  \item If $G_n$ denotes the graph of $u_n$, then the surface $S_n$
    obtained by doubling the conjugate graph $G_n^*$ has the same
    topological type as $\Sigma_n$.

  \item If $q_{i_1,\ldots,i_n}$ is an interior vertex of $\Omega_n $
    and $x_{i_1,\ldots,i_n}$ is a point in $\partial
    \Omega_n(\delta_{n+1})$ with $$\mbox{\rm dist}_{\mathbb{H}^2}(q_{i_1,\ldots,i_n},
    x_{i_1,\ldots,i_n})=\delta_{n+1},$$ then the third coordinate of
    $(x_{i_1,\ldots,i_n},u(x_{i_1,\ldots,i_n}))^*$ is less than $1/n$,
    where $(x_{i_1,\ldots,i_n},u(x_{i_1,\ldots,i_n}))^*$ means the
    conjugate point in the conjugate graph $G_n^*$ corresponding to
    $(x_{i_1,\ldots,i_n},u(x_{i_1,\ldots,i_n}))$.

  \end{enumerate}
  The existence of such a sequence is obtained by using Lemma
  \ref{lem:main} in a recursive way: First, we take $(\Omega_1,u_1)$
  as an admissible pair, where $\Omega_1$ is an admissible geodesic
  quadrilateral.  We call $q_0,p_0,q_1,p_1$ the vertices of
  $\Omega_1$, with $p_0,p_1\in\partial_\infty\mathbb{H}^2$.  For the sake of
  clarity, we are going to construct the admissible domain $\Omega_2$.
  Take $R_2>0$ such that $B(R_2)$ contains $q_0,q_1$, and
  $\varepsilon_2\in(0,1/4)$.  We choose $\delta_2\in(0,1/4)$ small enough so
  that $\Omega_1(\delta_2)\cap\partial B(R_2)$ has two components.
  According to the notation we have introduced for the exhaustion
  $\Sigma_n$, $n \in\mathbb{N}$, we should add an interior vertex and two new
  ideal vertices around $p_j$: We apply Lemma~\ref{lem:main} to
  $\Omega_1, \varepsilon_2,\delta_2,R_2$ and $p_j$. We call them $q_{j,1}$ and
  $p_{j,0},p_{j,1}$, respectively.  The remain vertices $q_i,p_i,q_j$
  of $\Omega_1$ remains fixed, and we call them
  $q_{i,0},p_{i,0},q_{j,0}$.  The vertices of $\Omega_2$ are then
  $q_{i,0},p_{i,0},q_{j,0},p_{j,0},q_{j,1},p_{j,1}$, consecutively
  ordered.  Note that this action has {\em the topological effect of
    adding a pair of pants} to the surface obtained by doubling the
  conjugate graph. In order to see this, we call $\Gamma_{q_i}\stackrel{\rm def}{=} \{
  q_i\} \times\mathbb{R}$, $i=0,1$, the vertical lines contained in the graph
  of $u_1$, denoted by $G_1$.  Let $\Gamma^*_{q_0}$ and
  $\Gamma^*_{q_1}$ be the conjugate curves in $G^*_1$. By Theorem
  \ref{th:JS}, $\Gamma^*_{q_0}$ and $\Gamma^*_{q_1}$ are horizontal
  lines of symmetry placed at height zero.  Since $\Omega_1$ is
  convex, then we know by Theorem~\ref{th:krust} that $G_1^*$ is a
  vertical graph over a domain that we call $\Omega_1^*$.  Similarly,
  we denote by $\gamma^*_{p_i}$, $i=0,1$, the geodesics in $\partial
  \Omega^*_1$ given by Theorem \ref{th:JS}, where
  $u_1^*|_{\gamma^*_{p_i}}=+\infty$. When we reflect $G_1*$ with
  respect to the slice $\{t=0\}$ and obtain a properly embedded
  minimal surface $S_1$ with genus zero and two ends. The ends are
  asymptotic to the vertical geodesic planes $\gamma^*_{p_i} \times
  \mathbb{R}$. In this sense, we could say that there exists a natural
  correspondence between the ends of $S_1$ and the ideal vertices of
  $\Omega_1$, $p_0$ and $p_1$.  After the application of Lemma
  \ref{lem:main}, we are substituting the end associated to $p_j$ by
  two new ends; the ones associated to $p_{j,0}$ and $p_{j,1}$,
  respectively. These two new ends are linked by the horizontal curve
  of symmetry $\Gamma_{q_{j,1}}^*$ (see Figure \ref{fig:omega1}.)
  \vskip 3mm

\begin{figure}[htbp]
   \begin{center}
      \includegraphics[width=0.4\textwidth]{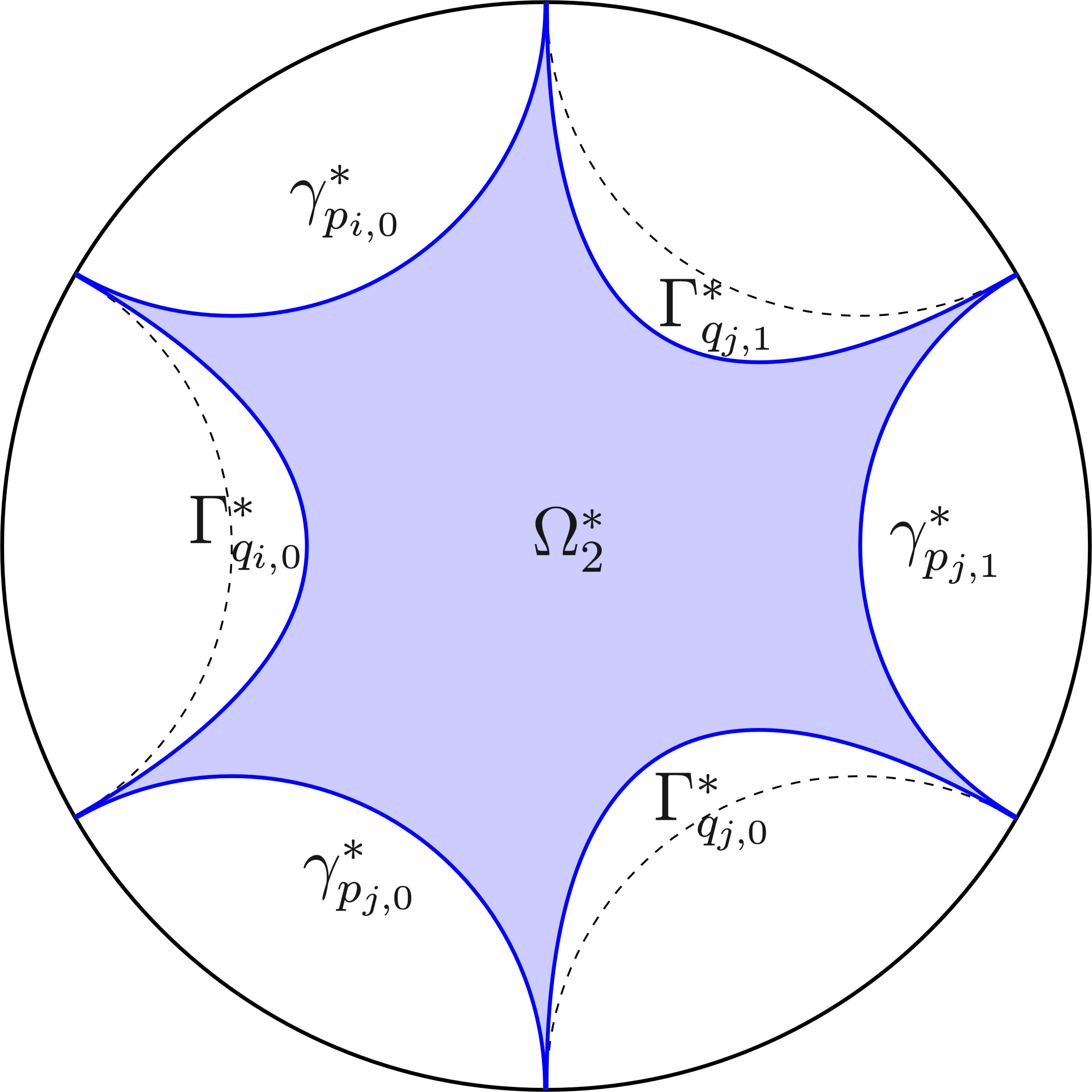}
   \end{center}
   \caption{The domain $\Omega_2^*$.}\label{fig:omega1}
 \end{figure}

  Now, assume we have $(\Omega_n,u_n)$ satisfying conditions above,
  and let us construct $(\Omega_{n+1},u_{n+1})$.  We fix $R_{n+1}>0$
  such that $B(R_{n+1})$ contains all the interior vertices of
  $\Omega_n$.  We choose $\delta_{n+1}\in(0,1/2^{n+1})$ small enough
  so that $\Omega_n(\delta_{n+1})\cap\partial B(R_{n+1})$ has $n+1$
  components.  We also take $\varepsilon_{n+1}\in(0,1/2^{n+1})$. As above, the
  effect of adding a pair of pants to the boundary $\partial_{j_1,
    \ldots,j_n}$ of $\Sigma_n$ means that we have to substitute the
  ideal vertex $p_ {j_1, \ldots,j_n}$ by two new ideal vertices, that
  we will call $p_ {j_1, \ldots,j_n,0}$ and $p_ {j_1,
    \ldots,j_n,1}$. To do this we apply, as before,
  Lemma~\ref{lem:main} to: $\Omega_n, \varepsilon_{n+1},\delta_{n+1},R_{n+1}$
  and $p_ {j_1, \ldots,j_n}$. A new interior vertex also appears, we
  call it $q_ {j_1, \ldots,j_n,1}$. Finally, we relabel the other
  vertices just by adding a $0$ in the subindex.

  Let us define $\displaystyle \Omega \stackrel{\rm def}{=}\bigcup_{n=1}^\infty
  \Delta_n$. It is not hard to prove that $\displaystyle \Omega=
  \bigcup_{n=1}^\infty \left( \Omega_n \cap B(R_{n+1}) \right)$ and
 $\Omega$ is {\it convex.}

 Taking into account that the sequence $\{u_n\}_{n \in\mathbb{N}}$ satisfies
 item (e) and that $\sum_n \varepsilon_n$ converges, then we obtain
 that $\{u_n\}_{n \in\mathbb{N}}$ is a Cauchy sequence, with respect to the
 smooth convergence on compact sets in $\Omega$. Ascoli-Arcela's
 theorem implies that $\{u_n\}_{n \in\mathbb{N}}$ converges to a smooth
 function $u$ which is also a solution of (\ref{eq.min.surf}) on
 $\Omega$. Label the graph surface of $u$ as $G$.  As $\Omega$ is
 convex, then Theorem \ref{th:krust} says us that $G^*$ is a also a
 graph over a domain that we call $\Omega^*$. In particular, $G^*$ is
 embedded.

 \begin{claim} \label{claim:graph} The limit graph $G$ contains
   vertical straight lines placed over the interior vertices of
   $\Omega_n$, for all $n\in\mathbb{N}$.
 \end{claim}
 In order to prove this claim, we fix $n_0 \in\mathbb{N}$ and let $q$ be a
 (fixed) interior vertex of $\Omega_{n_0}$. Two geodesics in $\partial
 \Omega_{n_0}$ arrive at this point, denoted by $\gamma_{n_0}^+$ and
 $\gamma_{n_0}^-$, with the properties that ${u_{n_0}}_{|
   \gamma_{n_0}^\pm}=\pm \infty.$ Recall that $q$ is an interior
 vertex of $\Omega_n$, for all $n\geq n_0$. Consider the corresponding
 boundary geodesics $\gamma_{n}^+,\gamma_{n}^-$ in $\partial \Omega_n$
 with ${u_{n}}_{| \gamma_{n}^\pm}=\pm \infty.$

 First, we focus on the sequence $\{\gamma_n^+\}_{n\in\mathbb{N}}$. Notice
 that, from the way in which we have obtained our sequence
 $\{\Omega_n\}_{n \in\mathbb{N}}$, the initial conditions of the geodesic
 $\gamma_n^+$ are given by $\gamma_n^+(0)=q,$ $(\gamma_n^+)'(0)={\rm
   e}^{{\rm i} \theta_n},$ where the sequence of arguments
 $\{\theta_n\}_{n \in\mathbb{N}}$ is monotone and bounded. So,
 $\{\theta_n\}_{n \in\mathbb{N}}$ converges to a real number $\theta$. Let
 $\gamma^+$ be the geodesic starting at $q$ with $(\gamma^+)'(0)={\rm
   e}^{{\rm i} \theta}.$ By construction, $\{\gamma_n^+\}_{n \in\mathbb{N}}$
 smoothly converges to $\gamma^+$. The geodesic $\gamma^+$ joins $q$
 with a point $p^+ \in \partial_\infty \mathbb{H}^2$. Moreover, $\gamma^+$ is
 part of $\partial \Omega$. Let $\rho^+$ be the radial geodesic
 arriving at $p^+$.  Taking our method of construction into account,
 we can guarantee that there are no interior vertices of $\Omega_n$,
 $n\geq n_0$, in the triangle $R^+$ whose sides consists of
 $\gamma^+$, a bounded piece of $\gamma_{n_0}^-$ starting at $q$ that
 we call $\sigma$ and a convex curve $\alpha$ (convex with respect to
 $R^+$) which is asymptotic to $\rho^+$ at $p^+$ (see Figure
 \ref{fig:r+}.) Let $v$ be the solution to the Dirichlet problem
 associated to equation \eqref{eq.min.surf} on $R^+$ with boundary
 data $+\infty$ on $\gamma^+$, $-\infty$ on $\sigma$ and
 $\mbox{inf}_{n \geq n_0} u_n$ on $\alpha$. Notice that $\mbox{inf}_{n
   \geq n_0} u_n$ is continuous over $\alpha$ and then solution $v$
 exists by Theorem 4.9 in~\cite{marr1}. Then, the generalized maximum
 principle given by Collin and Rosenberg in~\cite[Theorem 2]{cor2}
 (see also~\cite[Theorems 4.13 and 4.16]{marr1}) gives us that $v \leq
 u_n$ in $\Omega_n \cap R^+$, for all $n \geq n_0$.  This fact implies
 that $u_{|\gamma^+} =+\infty.$

\begin{figure}[htbp]
   \begin{center}
      \includegraphics[width=0.5\textwidth]{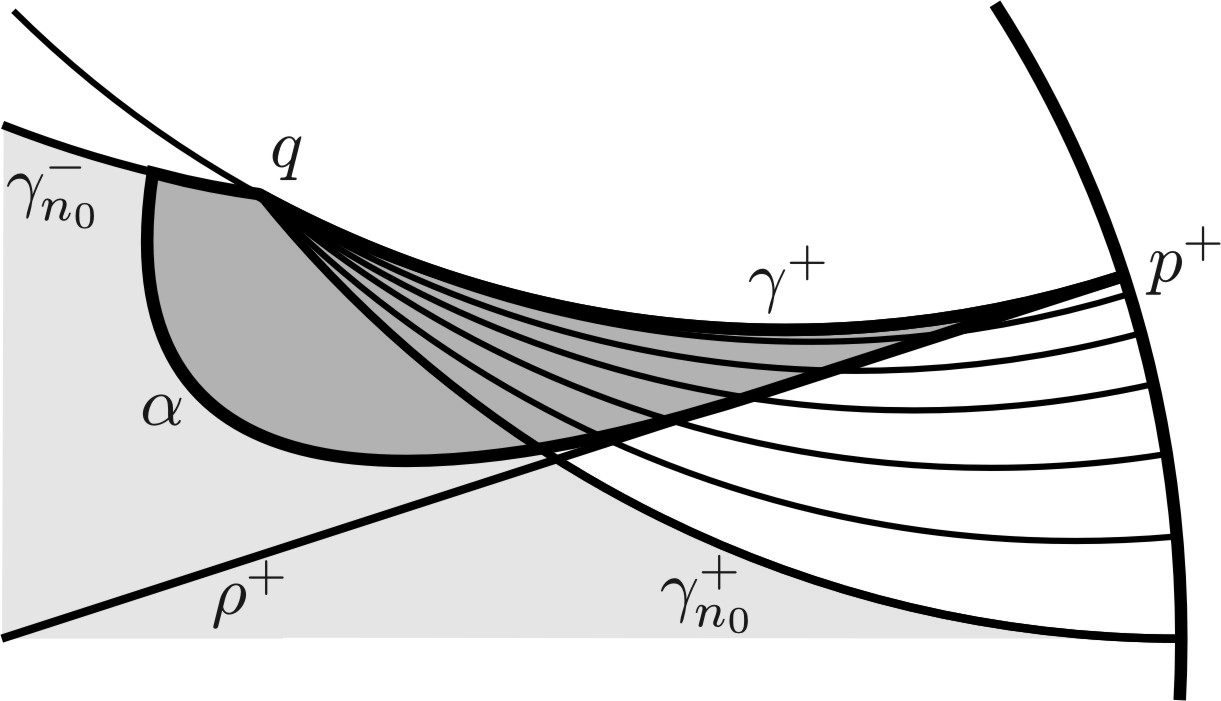}
   \end{center}
   \caption{}\label{fig:r+}
   \label{fig:barrera_recortado}
\end{figure}

A similar argument gives us that $u_{|\gamma^-} =-\infty,$ where
$\gamma^-$ is the limit of the sequence $\{\gamma_{n}^-\}.$ So, the
graph of $u$ extends to a vertical line over the point $q$.  This
concludes the proof of Claim \ref{claim:graph}.

Let $q$ be an interior vertex of $\Omega$ and $\Gamma_q
\stackrel{\rm def}{=} \{q \} \times\mathbb{R}$ the vertical line
contained in the graph of $u$, called $M$. Then, the conjugate curve
$\Gamma_q^*\subset M^*$ is a horizontal curvature line of symmetry
(see Subsection \ref{subsec:conj}.)

\begin{claim}
  \label{claim:altura}
  For any interior vertex $q$ in $\Omega$, $\Gamma_q^*$ is contained
  in the plane $\{t=0\}$. In particular, we can see $\Gamma_q^*$ as a
  part of $\partial \Omega^*$. In this sense, $\Gamma_q^*$ is concave
  with respect to $\Omega^*$. Moreover, the endpoints of $\Gamma_q^*$
  in $\partial_\infty \mathbb{H}^2$ are distinct.
\end{claim}

In order to prove this claim, we assume that $q$ is an interior vertex
of $\Omega_n$, for $n \geq k$.  As a vertex of $\Omega_n$, $q$ appears
represented as $q_{i_0, \ldots,i_n}$, with $i_j \in \{0,1\}$, $j=1,
\ldots,n.$ Let $x_{i_0, \ldots,i_n}$ be the corresponding point given
by item (g). By construction, the sequence $\{x_{i_0, \ldots,i_n}\}_{n
  \in\mathbb{N}}$ converges to $q$. So, $\{(x_{i_0, \ldots,i_n}, u(x_{i_0,
  \ldots,i_n}))^*\}_{n \in\mathbb{N}}$ is a sequence of points in $\mathbb{H}^2\times
\mathbb{R}$ accumulating to $\Gamma_q^*$. Taking item (g) into
account (and using that the intrinsic distance between two vertical
geodesics in the boundary of the graphs $G_n$ remains uniformly
bounded), this means that $\Gamma_q^*$ is contained in the slice
$\{t=0\}$, for any $q$. The concavity of $\Gamma_q^*$ with respect
to $\Omega^*$ is a simple consequence of the maximum principle for
minimal surfaces, using that $\Gamma_q^*$ is a curve of symmetry.

Now, we are going to see that the endpoints of $\Gamma_q^*$ are
distinct. We proceed by contradiction. We suppose that both branches
of $\Gamma_q^*$ arrive to the same ideal point $d
\in\partial_\infty\mathbb{H}^2$, and let $\sigma_\varepsilon$ be the
geodesic in $\mathbb{H}^2$ whose endpoints $d_\varepsilon^\pm$ are
disposed symmetrically in $\partial_\infty\mathbb{H}^2$ with respect
to $d$ and such that $\mbox{\rm
dist}_{\mathbb{R}^2}(d,d_\varepsilon^\pm)=\varepsilon$. We consider
the bounded convex region ${\cal D}$ in $\mathbb{H}^2$ bounded by
$\Gamma_q^*$ and $\sigma_\varepsilon$. If we apply Gauss-Bonnet
formula for $\varepsilon$ small enough, we obtain that
\begin{equation}\label{eq:kg}
  \mbox{Area}({\cal D})\leq \int_{\Gamma_q^*} k_g -\pi,
\end{equation}
where $k_g$ is the geodesic curvature of $\Gamma_q^*$ in
$\mathbb{H}^2$. Since the normal vector field of $M$ rotates less
than $\pi$ along $\Gamma_q$, we get $\int_{\Gamma_q^*} k_g\leq\pi$,
which contradicts~\eqref{eq:kg}.

We consider the closed set $D_n=\overline{ \Omega'_n} \cap \overline{
  B(R_{n+1})}$, where $\Omega_n'$ is a domain in $\Omega_n$ with the
same vertices than $\Omega_n$ joined by arcs which are contained in
$\Omega_n \setminus \Omega_n(\delta_n)$.  Denote by $M_n$ the graph
of $u$ over $D_n.$ $M_n$ is a minimal surface whose boundary
contains vertical segments over the interior vertices of $\Omega_n$.
Then the conjugate surface $M_n^*$ can be reflected with respect to
the horizontal slice $\mathbb{H}^2\times\{0\}$, and we obtain a
surface $S_n$ which is homeomorphic to $\Sigma_n$. Furthermore, if
we label $f_n : \Sigma_n \to S_n$ to this homeomorphism, we have for
all $i\leq n$ that ${f_n}|_{\Sigma_i}$ coincides with the
corresponding homeomorphism $f_i:\Sigma_i \to S_i$, since
$D_i\subset D_n$.

Let $S$ be complete surface obtained by gluing together both $G^*$
and its reflection with respect to $\mathbb{H}^2\times\{0\}$. We
have that $S_n$ is a simple exhaustion of $S$ and the sequence of
homeomorphisms $\{f_n \}_{n \in\mathbb{N}}$ has a limit $f:\Sigma
\to S$. \vskip 3mm

In order to prove item (3) in the statement of the theorem, we
consider $E_1$ and $E_2$ two different ends of $f(\Sigma)$. Then
there is a first natural $n \in\mathbb{N}$ so that $E_1$ and $E_2$
are represented by two different components of $\Sigma-
\left(\cup_{i=1}^n \Sigma_i \right).$ This is $\partial_{i_1,
\ldots, i_n}$ is the boundary of a component representing both ends
$E_1$ and $E_2$, but $\partial_{i_1,
  \ldots,i_{n}, 0}$ represents $E_1$ and $\partial_{i_1, \ldots,i_{n},
  1}$ represents $E_2$. Consider the points $q_1=q_{i_1,
  \ldots,i_{n},0}$ and $q_2=q_{i_1, \ldots,i_{n},1}$ which are
interior vertices of $\Omega$.  From Claim \ref{claim:altura} we
know that $\Gamma_{q_1}^*$ and $\Gamma_{q_2}^*$ are curves in
$\partial \Omega^*$ with distinct endpoints. Moreover, these two
curves cannot be asymptotic. Let $\eta_1$ and $\eta_2$ be the
geodesics in $\mathbb{H}^2$ joining an end point of $\Gamma_{q_1}^*$
to an endpoint of $\Gamma_{q_2}^*$ in such a way that $\eta_1 \cup
\Gamma_{q_1}^* \cup \eta_2 \cup \Gamma_{q_2}^*$ bounds an open ideal
quadrilateral ${\cal
  Q}$. Hence, the limit sets $L(E_1)$ and $L(E_2)$ lie in different
components of $\partial_\infty ( (\mathbb{H}^2-{\cal Q})
\times\mathbb{R}).$
\end{proof}

Finally, we would like to discuss about the underlying conformal structure
of the minimal surfaces we have just constructed. A good reference for the
notation and results we are going to use is \cite[\S 6 and \S 15]{ahlfors}.

As we have mentioned before, it is important to note that if
$\Sigma$ has a finite number of ends, then the examples provided in
the above theorem are those already constructed by Morabito and the
second author. These examples have total curvature $-4 \pi(k-1)$,
where $k$ represents the number of ends. Thus, using a classical
result by Huber \cite{huber}, Morabito-Rodr\'\i guez's surfaces are
conformally equivalent to a sphere minus $k$ points.  In particular,
they are parabolic (see definition below).  The examples with
infinite topology given by Theorem \ref{th:main} no longer have
finite total curvature. However, we would like to point out that
they can be constructed with parabolic conformal type, as explained
in Remark~\ref{rem:parabolic}.

\begin{definition}
\label{def:parabolic}
An open Riemann surface $W$ is said to be parabolic if there are no
non-constant negative subharmonic functions on $W$.
\end{definition}

Among other important characterizations of parabolicity, we
know that $W$ is parabolic if and only if one of the following conditions is
fulfilled:
\begin{itemize}
\item the maximum principle for harmonic maps is valid on $W$;
\item the harmonic measure of the ideal boundary of $W$ vanishes;
\item there is no Green's function defined on $W$.
\end{itemize}

\begin{remark}\label{rem:parabolic}
  The embedding $f:\Sigma \to \mathbb{H}^2\times\mathbb{R}$ in Theorem \ref{th:main}
  can be constructed in such a way that $f(\Sigma)$ is parabolic.  To
  do this, we consider the simple exhaustion $$S_1 \subset S_2 \subset
  \cdots \subset S_n \subset \cdots$$ given in the proof of the
  theorem. We denote by $\lambda_n$ the {\em extremal length} between
  $\partial S_1$ and $\partial S_n$ and by $\mu_n$ the harmonic
  modulus $\mu_n \stackrel{\rm def}{=} {\rm e}^{\lambda_n}$.  Notice that the surface
  obtained by doubling the graph $G_n^*$ is parabolic (it has finite
  total curvature). So, using Lemma \ref{lem:main} in a suitable way,
  we could guarantee in our inductive process that $\mu_n \geq n-1$.
  This fact implies that $S=f(\Sigma)$ is parabolic.
\end{remark}

\bibliographystyle{plain}

\begin{thebibliography}{10}

\bibitem{ahlfors} {\em L.V. Ahlfors and L. Sario}, Riemann Surfaces. {\em Princeton University Press}, 1960.
\bibitem{cmCourant}
{\em T.~H. Colding and W.~P. Minicozzi~II},
\newblock {Minimal surfaces}, volume~4 of {Courant Lecture Notes in
  Mathematics},
\newblock New York University Courant Institute of Mathematical Sciences, New
  York (1999).

\bibitem{cm34}
{\em T.~H. Colding and W.~P. Minicozzi~II},
\newblock An excursion into geometric analysis,
\newblock in {Surveys of Differential Geometry IX - Eigenvalues of
  Laplacian and other geometric operators}, pages 83--146. International Press,
  edited by Alexander Grigor'yan and Shing Tung Yau (2004).

\bibitem{cor2}
{\em P.~Collin and H.~Rosenberg},
\newblock Construction of harmonic diffeomorphisms and minimal
graphs.
\newblock {Annals of Math.}, {\bf 172} (2010), 1879--1906.

\bibitem{da2}
{\em B.~Daniel},
\newblock Isometric immersions into {$\mathbb{S}^n\times\mathbb{R}$} and
  {$\mathbb{H}^n\times\mathbb{R}$} and applications to minimal
  surfaces.
\newblock {Trans. Amer. Math. Soc.}, {\bf 361} (2009), 6255--6282.

\bibitem{fmm}
{\em L. Ferrer, F. Mart\'\i n and W.~H. Meeks~III},
\newblock Existence of proper minimal surfaces of arbitrary
topological type.
\newblock Preprint.


\bibitem{hau1}
{\em L.~Hauswirth},
\newblock Minimal surfaces of riemann type in three-dimensional product
  manifolds.
\newblock {Pacific Journal of Math.}, {\bf 224} (2006), 91--117.

\bibitem{HST}
{\em L.~Hauswirth, R.~Sa Earp and E.~Toubiana},
\newblock Associate and conjugate minimal immersions $\mathbb{H}^2\times
  \mathbb{R}$.
\newblock {Tohoku Math. J.}, {\bf 60} (2008), 267--286.

\bibitem{huber} {\em A. Huber}, { On subharmonic functions and differential geometry
in the large.} Comment. Math. Helv., Vol. {\bf 32} (1957), 13-72.

\bibitem{JS}
{\em H. Jenkins and J. Serrin,}
 The Dirichlet problem for the minimal surface equation, with infinite data.
 Bull. Amer. Math. Soc. 72 1966 102-106.

\bibitem{marr1}
{\em L.~Mazet, M.M.~Rodr\'\i guez, and H.~Rosenberg},
\newblock The {D}irichlet problem for the minimal surface equation with
  possible infinite boundary data over domains in a Riemannian
  surface.
\newblock Proc. London Math. Soc., {\bf 102} (2011), 985-1023.

\bibitem{mpe3}
{\em W.~H. Meeks~III and J.~P\'{e}rez},
\newblock Embedded minimal surfaces of finite topology.
\newblock Preprint, available at {\tt
  http://www.ugr.es/local/jperez/papers/papers.htm}.

\bibitem{mpe1}
{\em W.~H. Meeks~III and J.~P\'{e}rez},
\newblock Conformal properties in classical minimal surface theory,
\newblock in {Surveys of Differential Geometry IX - Eigenvalues of
  Laplacian and other geometric operators}, pages 275--336. International
  Press, edited by Alexander Grigor'yan and Shing Tung Yau, 2004.

\bibitem{mpr6}
{\em W.~H. Meeks~III, J.~P\'{e}rez, and A.~Ros},
\newblock Properly embedded minimal planar domains,
\newblock preprint, available at {\tt
  http://www.ugr.es/local/jperez/papers/papers.htm}.

\bibitem{moro1}
{\em F.~Morabito and M.M.~Rodr\'\i guez},
\newblock Saddle towers and minimal $k$-noids in
  $\mathbb{H}^2\times\mathbb{R}$.
\newblock To appear in J. Inst. Math. Jussieu, arXiv:math/0910.5676.

\bibitem{ner2}
{\em B.~Nelli and H.~Rosenberg},
\newblock Minimal surfaces in $\mathbb{H}^2 \times \mathbb{R}$.
\newblock {Bull. Braz. Math. Soc.}, {\bf 33} (2002), 263--292.
\newblock MR1940353, Zbl 1038.53011.

\bibitem{pyo1}
{\em J.~Pyo},
\newblock New complete embedded minimal surfaces in
  $\mathbb{H}^2\times\mathbb{R}$.
\newblock Preprint, arXiv:math/0911.5577.

\bibitem{R}
{\em M.M.~Rodr\'\i guez},
\newblock Minimal surfaces with limit ends in $\mathbb{H}^2\times\mathbb{R}$.
\newblock Preprint, arXiv:math/1009.3524.

\end{thebibliography}

\mbox{}\\

\noindent
Francisco Mart\'\i n and M. Magdalena Rodr\'\i guez\\
Departamento de Geometr\'\i a y Topolog\'\i a\\
Universidad de Granada\\
Fuentenueva, 18071, Granada, Spain\\
e-mail: \texttt{fmartin@ugr.es, magdarp@ugr.es}

\end{document}